\documentclass[journal]{IEEEtran}
\usepackage{tikz}
\usetikzlibrary{shapes, shapes.geometric, shapes.symbols, shapes.arrows, shapes.multipart, shapes.callouts, shapes.misc}
\usepackage[activate={true,nocompatibility},final,tracking=true,kerning=true,spacing=true,factor=1100,stretch=10,shrink=10]{microtype}

\usepackage{etex}
\usepackage{graphicx}
\usepackage{amsmath}
\usepackage{amssymb}
\usepackage{mathrsfs}
\usepackage[vlined,ruled]{algorithm2e}
\usepackage{subfigure}
\usepackage{cite}
\usepackage{epstopdf}
\usepackage{booktabs}

\newtheorem{theorem}{Theorem}[section]
\newtheorem{lemma}[theorem]{Lemma}
\newtheorem{definition}{Definition}

\newtheorem{proposition}[theorem]{Proposition}

\newtheorem{remark}{Remark}

\newenvironment{pfof}[1]{\vspace{1ex}\noindent{\itshape Proof of
    #1:}\hspace{0.5em}} {\hfill\oprocend\vspace{1ex}}
\newenvironment{proof}[1]{\vspace{1ex}\noindent{\itshape Proof:}\hspace{0.5em}} {\hfill\oprocend\vspace{1ex}}

\newcommand*\mc[0]{\mathcal}                                                  		

\DeclareMathOperator*{\minimize}{minimize} 	

\newcommand{\until}[1]{\{1,\dots, #1\}}

\newcommand{\map}[3]{#1: #2 \rightarrow #3}

\newcommand{\real}{\mathbb{R}}
\newcommand{\complex}{\mathbb{C}}

\newcommand\oprocendsymbol{\hbox{$\square$}}
\newcommand\oprocend{\relax\ifmmode\else\unskip\hfill\fi\oprocendsymbol}

\newcommand{\vect}[1]{\mathbbold{#1}}
\newcommand{\vones}[1][]{\vect{1}_{#1}}
\newcommand{\vzeros}[1][]{\vect{0}_{#1}}
\DeclareSymbolFont{bbold}{U}{bbold}{m}{n}
\DeclareSymbolFontAlphabet{\mathbbold}{bbold}

\newcommand{\tb}{}

\graphicspath{{./Figures/}}

\usepackage{enumerate}

\usepackage{xcolor}

\ifCLASSINFOpdf
\else
\fi

\graphicspath{{./Figures/}}

\begin{document}
%
\title{Voltage Stabilization in Microgrids\\ via Quadratic Droop Control}
%
%
%

\author{John~W.~Simpson-Porco,~\IEEEmembership{Member,~IEEE,}
        Florian~D\"{o}rfler,~\IEEEmembership{Member,~IEEE,}
        and~Francesco~Bullo,~\IEEEmembership{Fellow,~IEEE}
\thanks{J.~W.~Simpson-Porco is with the Department of Electrical and Computer Engineering, University of Waterloo. Email: {\footnotesize jwsimpson@uwaterloo.ca}. F.~Bullo is with the Department of Mechanical Engineering and the Center for Control, Dynamical System and Computation, University of California at Santa Barbara. Email: 
   {\footnotesize bullo@engineering.ucsb.edu}. F. D\"orfler is with the Automatic Control Laboratory, ETH Z\"{u}rich. Email: 
{dorfler@ethz.ch}.}
\thanks{This work was supported by NSF grants CNS-1135819 and by the National Science and Engineering Research Council of Canada.}
}

%
%

\markboth{}%
{Shell \MakeLowercase{\textit{et al.}}: Bare Demo of IEEEtran.cls for Journals}
%



\maketitle



\begin{abstract}
We consider the problem of voltage stability and reactive power balancing in {\tb islanded} small-scale electrical networks outfitted with DC/AC inverters (``microgrids'').
A droop-like voltage feedback controller is proposed which is quadratic in the local voltage magnitude, allowing for the application of circuit-theoretic analysis techniques to the closed-loop system. 
{\tb The operating points of the closed-loop microgrid are in exact correspondence with the solutions of a reduced power flow equation, and we provide explicit solutions and small-signal stability analyses under several static and dynamic load models.}
Controller optimality is characterized as follows: we show a one-to-one correspondence between the high-voltage equilibrium of the microgrid under quadratic droop control, and the solution of an optimization problem which minimizes {\tb a trade-off between} reactive power dissipation and voltage deviations. 
Power sharing performance of the controller is characterized as a function of the controller gains, network topology, and parameters. Perhaps surprisingly, proportional sharing of the total load between inverters is achieved in the low-gain limit, independent of the circuit topology or reactances.
All results hold for arbitrary grid topologies, with arbitrary numbers of inverters and loads. {\tb Numerical results confirm the robustness of the controller to unmodeled dynamics.}
\end{abstract}



\section{Introduction}
\label{Section: Introduction}


The wide-spread integration of low-voltage small-scale renewable energy sources requires that the present centralized electric power transmission paradigm to evolve towards a more distributed future. 
As a flexible bridge between {\tb distributed generators} and larger distribution grids, \emph{microgrids} continue to attract attention \cite{JAPL-CLM-AGM:06,JMG-JCV-JM-LGDV-MC:11,DEO-AMS-AHE-CAC-RI-MK-AHH-OGM-MS-RPB-GAJE-NDH:14}.
Microgrids are low-voltage electrical distribution networks, heterogeneously composed of distributed generation, storage, load, and managed autonomously from the larger primary grid. 
While often connected to the larger grid through a ``point of common coupling'', microgrids are also able to ``island'' themselves and operate independently \cite{JMG-JCV-JM-LGDV-MC:11,QCZ-TH:13}. This independent self-sufficiency is crucial for reliable power delivery in remote communities, in military outposts, in developing nations lacking large-scale infrastructure, and in backup systems for critical loads such as hospitals and campuses.
Energy generation within a microgrid can be quite heterogeneous, including photovoltaic, wind, micro-turbines, etc. Many of these sources generate either variable frequency AC power or DC power, and are interfaced with a synchronous AC grid via power electronic DC/AC \emph{inverters}.
In islanded operation, {\tb at least some} of these inverters must operate as \emph{grid-forming} devices. That is, control actions must be taken through them to ensure synchronization, voltage stability, and load sharing in the network \cite{JAPL-CLM-AGM:06,FK-RI-NH-AD:08,JMG-JCV-JM-LGDV-MC:11,QCZ-TH:13}, and to establish higher-level objectives such as frequency regulation and economic dispatch \cite{jwsp-fd-fb:12u,jwsp-qs-fd-jmv-jmg-fb:13s,FD-JWSP-FB:13y}.

\begin{figure}[t]
\begin{center}
\subfigure[]{\includegraphics[height=0.41\columnwidth]{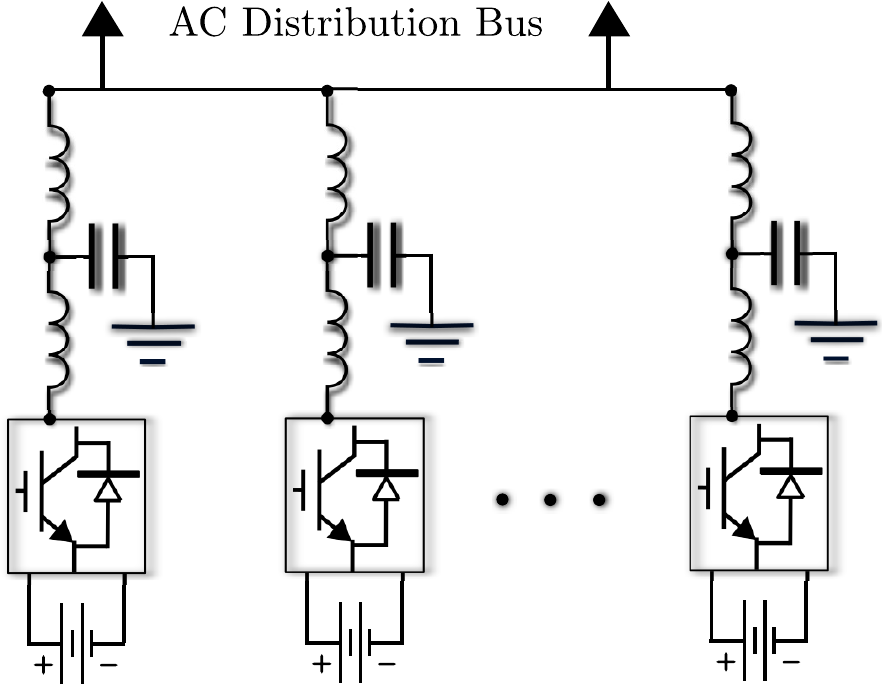}\label{Fig:InvNet}}\,\,\,\,\,
\subfigure[]{\includegraphics[width=0.4\columnwidth]{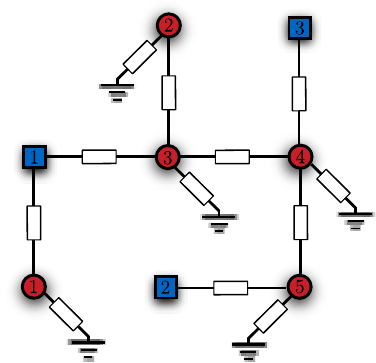}\label{Fig:Circuit1}}
\caption[(a) Schematic of a ``parallel'' microgrid, in which several inverters supply power to a distribution bus (effectively a single load) (b) A simple non-parallel microgrid consisting of five loads and three inverters.]{(a) Schematic of a ``parallel'' microgrid, in which several inverters supply power to a distribution bus (effectively a single load) (b) A simple non-parallel microgrid consisting of five loads {\tikz\draw[black,semithick,fill=red] (0,0) circle (.5ex);} and three inverters {\tikz{\path[draw=black,semithick,fill={rgb:red,0;green,102;blue,196}] (0,0) rectangle (0.12cm,0.12cm);}}.}
\label{Fig:ParallelandGeneral}
\end{center}
\end{figure}

\subsection{Literature Review}

The so-called \emph{droop controllers} (and their many derivatives) have been used with some success to achieve {\tb primary} control goals such as stability and load sharing, see \cite{AT-HJ-TU-KM:97,JMG-JCV-JM-LGDV-MC:11,SB-MC-PP-DP:02,JAPL-CLM-AGM:06,YWL-CNK:09,JMG-JCV-JM-MC-LGDV:09,QCZ-TH:13}. 
Despite being the foundational technique for networked operation of {\tb islanded microgrids} (Figure \ref{Fig:ParallelandGeneral}), the stability and fundamental limitations of droop-controlled microgrids has only recently begun to be investigated from a rigorous system-theoretic point of view \cite{jwsp-fd-fb:12u,JS-RO-AA-JR-TS:13,ZW-MX-MDL:13,AB-FLL-AD:14}. 
%

{\tb Our focus here is on voltage control, which we now provide some context for. In high-voltage networks, the grid-side voltages of transformers which interface synchronous generators are regulated to nominal set points via automatic voltage regulator systems. }
However, if this strategy was applied to inverters in low-voltage networks, the small mismatches in voltage set points would combine with the large impedances presented by distribution lines, and result in large circulating reactive currents between inverters. Moreover, due to the close electrical proximity of devices in small microgrids, and the small power capacities (ratings) of the sources, it is desirable that controllers establish set percentages of the total load to be supplied by each inverter (so-called \emph{power sharing}).
These technical obstacles 
motivated
the use of voltage droop control \cite{MCC-DMD-RA:93} as a heuristic proportional controller to establish power sharing between units while maintaining voltages within a reasonable range around their set point values; {\tb a technical review of the voltage droop controller is presented in Section \ref{Subsection: Review of Conventional Droop Controller}.}
%


The widespread deployment of voltage droop has led to several attempts at stability analysis \cite{ZW-MX-MDL:13,JS-RO-AA-JR-TS:13}. 
Both references begin with all-to-all Kron-reduced network models which do not explicitly contain loads. 
They then assume the existence of a system equilibrium, and derive sufficient stability conditions which depend on the uncharacterized equilibrium.
Results apply only for constant-impedance or constant-current load models, and no characterizations are given of reactive power sharing in steady state. 
Thus, the available literature offers no guidance on the foundational issue of operating point feasibility, i.e., the 
existence and locations of steady states for the microgrid which satisfy operational constraints.
This gap of knowledge means that precise reactive loading limits and security margins are unknown, making system monitoring and non-conservative operation difficult.
{\tb Moreover, stable network operation is often limited by stiff constant power loads, which do not reduce their current consumption when voltages fall.}
%


Regarding control performance, a well-recognized drawback of voltage droop control is that \textemdash{} while achieving better power sharing performance than mere voltage regulation \textemdash{} the power sharing properties can still be quite poor. This has led some authors to investigate alternative strategies to implement reactive power sharing \cite{YWL-CNK:09,QCZ:13,JS-TS-JR-TS:15,jwsp-qs-fd-jmv-jmg-fb:13s} that are based on directly load measurements or inter-unit communication. Aside from these additional requirements, these approaches do not assist in characterizing the power sharing properties {\tb of the standard, decentralized droop control.}
A general discussion of fundamental limits for voltage control and reactive power sharing can be found in \cite[Section III]{jwsp-qs-fd-jmv-jmg-fb:13s}, and some recent non-droop-based control (but nonetheless related) approaches to reactive power compensation can be found in \cite{KT-PS-SV-MC:11,SB-SZ:12, MF-RN-CC-SL:12, MF-CL-SL:13,DOE-SG:12,NL-GQ-MD:14,AHE-EJD-RI:12,SR-FS-GFT:15,VN-QS-JMG-FLL-AD:16,MT-JWSP-FD-RC-FB:15a}. 

In summary, two fundamental outstanding problems regarding the voltage stability of droop-controlled microgrids are (i) {\tb to establish conditions under which a stable network equilibrium exists for various load models, and (ii) to characterize the resulting power sharing properties in terms of grid topology, branch admittances, loads, and controller gains.}

\subsection{Contributions}

In the preliminary version of this work \cite{jwsp-fd-fb:13h}, we introduced the quadratic droop controller, a slight modification of the conventional linear droop curve to a quadratic one. Quadratic droop leads to the same steady states as conventional droop control (for particular gains) but allows for circuit-theoretic techniques to be leveraged for analysis purposes. In \cite{jwsp-fd-fb:13h} we studied system stability for the case of {\tb an islanded parallel} microgrid, where several inverters feed power to a single common load or a common distribution bus (Remark \ref{Rem:Parallel}). In this paper we 
depart from the case of a parallel microgrid and study system stability for arbitrary interconnections of inverters and loads, {\tb under various load models}. In addition, we investigate the optimality and performance characteristics of the controller. {\tb Linearization} ideas presented in \cite{BG-JWSP-FD-SZ-FB:13zb} are also relevant to the current presentation at a technical level, and will be used in proofs of several results.


There are three main technical contributions in this paper. First, in Section \ref{Section: Quadratic Droop Control for Voltage Stabilization} we present 
 the quadratic droop controller, highlighting its circuit-theoretic foundations. We state and prove a correspondence between the equilibrium points of the closed-loop system and the solutions of a reduced power flow equation, before proceeding to analyze the stability of the closed-loop system for both static and dynamic load models. In particular, we analyze system stability for the static ZI and ZIP load models, and for a dynamic model describing a variable shunt susceptance.  Our stability analyses in Theorem \ref{Thm:ReducedFixedPoint}, Theorem \ref{Lem:ZILoads}, and Theorem \ref{Thm:ZIP} result in succinct and physically intuitive stability conditions, significantly generalizing the results and intuition developed for parallel microgrids in \cite{jwsp-fd-fb:13h}.


Second, in Section \ref{Sec:Optimal} we investigate the optimality properties of the quadratic droop controller. We demonstrate that the microgrid dynamics under quadratic droop control can be interpreted as a decentralized primal algorithm for solving a particular optimization problem, leading to an additional interpretation of the quadratic droop controller as an algorithm which attempts to {\tb minimize a trade-off between} total reactive power dissipation and voltage deviations. We comment on possible extensions and generalizations of these results.

Third and finally, in Section \ref{Sec:PowerShare} we investigate the power sharing properties of our controller, and closely examine two asymptotic limits. In the limit of high controller gains, we show that the reactive power is supplied by an inverter to a load in inverse proportion to the electrical distance between the two, independent of relative controller gains. Conversely, in the low-gain limit we find that inverters provide reactive power in proportion to their controller gains, independent of electrical distance. When neither limit is appropriate, we provide a general formula and quantify the power sharing error as a function of the controller gains and network parameters. {\tb These predictions along with the robustness of the approach are studied numerically in Section \ref{Sec:Sim}.}
%
%
%
{
Remarks throughout the paper provide details on extensions and interpretations of our results, and can be skipped on a first reading without loss of continuity}.

 
At a more conceptual level, this work has two major features which distinguish it from the primary literature on droop control. The first is our analytical approach, based on circuit theory \cite{fd-fb:11d} and approximations of nonlinear circuit equations \cite{BG-JWSP-FD-SZ-FB:13zb}. While we make specific modeling assumptions, they are standard in the field \cite{QCZ-TH:13} and allow us to derive strong yet intuitive results regarding stability and optimality of operating points. Our focus is on characterizing the operating points arising from decentralized voltage control, along with characterizing controller performance at the network level. We make use of different load models in different sections of the paper, in part to simplify certain presentations, and in part to illustrate a variety of analysis techniques which should prove useful for other researchers.
The second distinguishing feature is our explicit consideration of loads which are not collocated with inverters. Voltages at non-collocated loads are typically lower then the controlled voltages at inverters, and hence it is these load voltages which ultimately limit network stability. 

The remainder of this section recalls some basic notation. Section \ref{Section: Problem Setup and Review of AC Power Flow} reviews basic models for microgrids, loads, and inverters, along with the conventional droop controller and the relevant literature. Our main results are housed in Sections \ref{Section: Quadratic Droop Control for Voltage Stabilization}--\ref{Sec:Sim}. In Section \ref{Section: Conclusions} we summarize and provide directions for future research.

%
%
%
%
%
%
%
%
%
%

\subsection{Preliminaries and Notation}
{\it Sets, vectors and functions:} We let $\real$ (resp. $\real_{>0}$) denote the set of real (resp. strictly positive real) numbers, {\tb and let $\boldsymbol{\mathrm{j}}$ be the imaginary unit.}
Given $x \in \real^{n}$, $[x] \in \real^{n\times n}$ is the associated diagonal matrix with $x$ on the diagonal.
%
Throughout, $\vones[n]$ and $\vzeros[n]$ are the $n$-dimensional vectors of unit and zero entries, and $\vzeros[]$ is a matrix of all zeros of appropriate dimensions. The $n \times n$ identity matrix is $I_n$. A matrix $W \in \real^{n \times m}$ is row-stochastic if its elements are nonnegative and $W\vones[m] \!=\! \vones[n]$. 
%
A matrix $A \in \real^{n \times n}$ is an $M$-matrix if $A_{ij} \leq 0$ for all $i \neq j$ and all eigenvalues of $A$ have positive real parts. In this case $-DA$ is Hurwitz for any positive definite diagonal $D \in \real^{n \times n}$ (so-called $D$-stability), and $A^{-1} \geq \vzeros[]$ component-wise, with strict inequality if $A$ is irreducible \cite{AB-RJP:94}.

\section{Modeling and Problem Setup}
\label{Section: Problem Setup and Review of AC Power Flow}

We briefly review our microgrid model and recall the conventional voltage droop controller.

\subsection{Review of Microgrids and AC Circuits}
\label{Subsec:Review}

\paragraph{Network Modeling} We adopt the standard model of a quasi-synchronous microgrid as a linear circuit represented by a connected and weighted graph $G(\mathcal{V},\mc E)$, where $\mc V = \until {n+m}$ is the set of vertices (buses) and $\mc E \subseteq \mc V \times \mc V$ is the set of edges (branches). There are two types of buses: $n \geq 1$ load buses $\mathcal{L} = \{1,\ldots,n\}$, and $m \geq 1$ inverter buses $\mathcal{I} = \{n+1,\ldots,n+m\}$, such that $\mathcal{V} = \mathcal{L} \cup \mathcal{I}$. These two sets of buses will receive distinct modeling treatments in what follows.

The edge weights of the graph are the associated branch admittances $y_{ij} = g_{ij} + \boldsymbol{\mathrm{j}}b_{ij} \in \complex$.\footnote[2]{For the short transmission lines of microgrids, line charging and leakage currents are neglected and branches are modeled as series impedances \cite{JM-JWB-JRB:08}. {\tb Shunt capacitors/reactors will be included later as load models.}}
The network is concisely represented by the symmetric bus admittance matrix $Y \in \complex^{(n+m)\times(n+m)}$, where the off-diagonal elements are $Y_{ij} = Y_{ji} = -y_{ij}$ for each branch $\{i,j\} \in \mathcal{E}$ (zero if $\{i,j\} \notin \mathcal{E}$), and the diagonal elements are given by {\tb $Y_{ii} = \sum_{j\neq i}^{n+m} y_{ij}$}. To each bus $i \in \mc V$ we associate a phasor voltage $V_i = E_ie^{\boldsymbol{\mathrm{j}}\theta_i}$ and a complex power injection $S_i = P_{i} + \boldsymbol{\mathrm{j}} Q_i$.
%
%
%
For dominantly inductive lines, transfer conductances may be neglected and $Y = \boldsymbol{\mathrm{j}}B$ is purely imaginary, where $B \in \real^{(n+m)\times(n+m)}$ is the \emph{susceptance} matrix. {We refer to Remark \ref{Rem:Decoupling} for a discussion of this assumption.} The power flow equations then relate the bus electrical power injections to the bus voltages via
\label{eq: power flow}
\begin{subequations}
\begin{align}
	P_{\mathrm{e},i} &= \sum\nolimits_{j=1}^{n+m} B_{ij} E_{i} E_{j} \sin(\theta_{i} - \theta_{j})
	\,, \;\;\quad i \in \mathcal{V}\,,
	\label{Eq:PFActive}\\
	\Bigl.
	Q_{\mathrm{e},i} &= - \sum\nolimits_{j=1}^{n+m} B_{ij} E_{i} E_{j} \cos(\theta_{i} - \theta_{j})
	\,, \;\; i \in \mathcal{V}.
	\label{Eq:PFReactive}
\end{align}
\end{subequations}
In this work we focus on dynamics associated with the reactive power flow equation \eqref{Eq:PFReactive}, and refer the reader to \cite{jwsp-fd-fb:12u,jwsp-qs-fd-jmv-jmg-fb:13s,FD-JWSP-FB:13y,BBJ-SVD-AOH-PTK:14} and the references therein for analyses of microgrid active power/frequency dynamics. We will work under the standard \emph{decoupling approximation}, where $|\theta_i - \theta_j| \approx 0$ and hence $\cos(\theta_i-\theta_j) \approx 1$ for each $\{i,j\} \in \mathcal{E}$; see \cite{TVC-CV:98,JMG-JCV-JM-MC-LGDV:09}. This can be relaxed to non-zero but constant power angles $\theta_i-\theta_j = \gamma_{ij} < \pi/2$ at the cost of {\tb more complicated formulae}, but we do not pursue this here; {\tb see Remark \ref{Rem:Decoupling}}. Under the decoupling assumption, the reactive power injection $Q_{\mathrm{e},i}(\theta,E)$ in \eqref{Eq:PFReactive} becomes a function of only the voltages $E$, yielding
\begin{equation}\label{Eq:RPFEDecoupled}
Q_{\mathrm{e},i}(E) = -E_i\sum_{j=1}^{n+m}\nolimits B_{ij}E_j\,,
\end{equation}
or compactly in vector notation as
\begin{equation}\label{Eq:QVec}
Q = -[E]BE\,,
\end{equation}
where $E = (E_1,\ldots,E_{n+m})$. For later reference, some well-known properties of the susceptance matrix are recorded in Lemma \ref{Lem:Properties}.
\medskip

\smallskip

\paragraph{Load Modeling} When not specified otherwise, we assume a static load model of the form $Q_i(E_i)$, where the reactive power injection $Q_i$ is a smooth function of the supplied voltage $E_i$. With our sign convention, $Q_i(E_i) < 0$ corresponds to an inductive load which consumes reactive power. Specific static and dynamic load models will be used throughout the paper, and we introduce these models when needed.

\paragraph{Inverter Modeling} {\tb A standard grid-side model for a smart inverter is as a controllable voltage source}
\begin{equation}\label{Eq:InverterModel}
\tau_i\dot{E}_i = u_i\,,\qquad i \in \mathcal{I}\,,
\end{equation}
where $u_i$ is a control input {to be designed, and $\tau_i > 0$ is a time constant accounting for sensing, processing, and actuation delays.} {\tb The model \eqref{Eq:InverterModel} assumes that the control loops which regulate the inverter's internal voltages and currents are stable, and that these internal loops are fast compared to the grid-side time scales over which loads change.} This model is widely adopted among experimentalists in the microgrid field \cite{JAPL-CLM-AGM:06,TCG-MP:07,JMG-MC-TL-PCL:13}, and further explanations can be found in \cite{jwsp-qs-fd-jmv-jmg-fb:13s} and references therein. We note that grid-feeding (i.e., maximum power point tracking) photovoltaic inverters or back-to-back converters interfacing wind turbines can be modeled from the grid side as constant power sources, and are therefore included in our framework as loads.

\subsection{Review of Conventional Droop Control}
\label{Subsection: Review of Conventional Droop Controller}
The \emph{voltage droop controller} is a decentralized controller for primary voltage control in islanded microgrids. The controller is a heuristic based on the previously discussed decoupling assumption, and has an extensive history of use; see \cite{MCC-DMD-RA:93,JMG-JCV-JM-LGDV-MC:11,AT-HJ-TU-KM:97,JAPL-CLM-AGM:06,YWL-CNK:09,QCZ:13}.
 For the case of inductive lines, the droop controller specifies the input signal $u_i$ in \eqref{Eq:InverterModel} as the local feedback  \cite[Chapter 19]{QCZ-TH:13}\footnote{For grid-connected inverters, one sometimes sees this formula augmented with power set points $u_i = -(E_i-E_i^*)-n_i(Q_{\mathrm{e},i}(E)-Q_i^{\rm set})$. Here we consider islanded operation where $Q_i^{\rm set} = 0$.}
%
%
\begin{equation}\label{Eq:DroopReactive}
u_i = -(E_i-E_i^*) - n_iQ_{\mathrm{e},i}(E)\,, \;\; \quad i \in \mathcal{I}\,,
\end{equation}
where $E_i^* > 0$ is the nominal voltage for the $i^{th}$ inverter, and $Q_{i}(E)$ is the measured reactive power injection; see \cite{ECF-LAA-LABT:08,QCZ-TH:13} for details regarding measurement of active/reactive powers.
%
The controller gain $n_i > 0$ is referred to as the droop coefficient.
From \eqref{Eq:DroopReactive}, it is clear that if the inverter reactive power injection $Q_{\mathrm{e},i}(E)$ is non-zero, the voltage $E_i$ will deviate from $E_i^*$.

\begin{remark}\textbf{(Comments on Modeling, {\tb Decoupling,} and Droop Control)}\label{Rem:Decoupling} 
{\tb
Formally, the decoupling assumption leading to \eqref{Eq:RPFEDecoupled} is an assumption about the \emph{sensitivity} of reactive power injections $Q_{\mathrm{e},i}$ with respect to phase angles $\theta_i-\theta_j$. From \eqref{Eq:PFReactive}, this sensitivity is proportional to $\sin(\theta_i-\theta_j)$, and around normal operating points is therefore roughly zero; phase angles effect reactive power only though second-order effects. The strong relationship between reactive power and voltage magnitudes is why the former is often used as a tool to regulate the latter.}
The voltage droop controller \eqref{Eq:DroopReactive} is designed under the assumption that the susceptance to conductance ratios $b_{ij}/g_{ij}$ are large within the microgrid. This assumption is typically typically justified in engineered settings, as the inverter output impedances are controlled to dominate over network impedances giving the network a strongly inductive characteristic \cite{JMG-LG-JM-MC-JM:05},\cite[Chapter 7]{QCZ-TH:13}. For dominantly resistive (or capacitive) microgrids, the appropriate droop controllers take different forms \cite[Chapter 19]{QCZ-TH:13}, with the following general structure: assuming all lines have uniform reactance/resistance ratios $b_{ij}/g_{ij} = $ const., the controller is modified by replacing $Q_{\mathrm{e},i}$ in \eqref{Eq:DroopReactive} with $\widetilde{Q}_{\mathrm{e},i} = \cos(\varphi)P_{\mathrm{e},i} + \sin(\varphi)Q_{\mathrm{e},i}$, where $\tan(\varphi) = -b_{ij}/g_{ij}$. A standard calculation \cite{FD-JWSP-FB:13y} yields $\widetilde{Q}_{\mathrm{e},i} = -\sum_{j=1}^{n+m}\widetilde{B}_{ij}E_iE_j\cos(\theta_i-\theta_j)$, where $\widetilde{B}$ is an admittance matrix with edge weights $-(b_{ij}^2+g_{ij}^2)^{1/2}$. This formula is of the same mathematical form as \eqref{Eq:PFReactive}, and for such uniform networks, we can restrict ourselves --- without loss of generality --- to the decoupled reactive power flow \eqref{Eq:QVec} and the conventional droop controller \eqref{Eq:DroopReactive}. {\tb As a special case of the above, we note that all analysis herein also applies to voltage/active power droop control for dominantly resistive microgrids.}
%
\hfill \oprocend \end{remark}

\smallskip

Despite its extensive history, the voltage droop controller \eqref{Eq:DroopReactive} has so far resisted any rigorous stability analysis. In our opinion, the key obstacle has been \textemdash{} and remains \textemdash{} the difficulty in determining the high-voltage equilibrium of the closed-loop system for arbitrary interconnections of devices. We remove this obstacle in the next section by proposing a controller based on the nonlinear physics of AC power flow.

\section{Quadratic Droop Control}
\label{Section: Quadratic Droop Control for Voltage Stabilization}

\subsection{Definition and Interpretation}

While the droop controller \eqref{Eq:DroopReactive} is simple and intuitive, it is based on the {\tb linearized} behavior of AC power flow around the system's open-circuit operating point, and does not respect the inherently quadratic nature of reactive power flow.
%
%
%
%
We instead propose a physically-motivated modification of the conventional voltage-droop controller \eqref{Eq:DroopReactive}. In place of \eqref{Eq:DroopReactive}, consider instead the \emph{quadratic droop controller}
\begin{equation}\label{Eq:VoltDroopQuadratic}
\boxed{\Bigl.
u_i = K_iE_i(E_i-E_i^*) - Q_{\mathrm{e},i}(E)\,,\,\quad i \in \mathcal{I}\,,
}
\end{equation}
where $K_i < 0$ is the controller gain.\footnote{While the chosen sign convention $K_i < 0$ is unconventional, it will simplify the resulting formulae and help us interpret the controller physically.} Compared to the conventional controller \eqref{Eq:DroopReactive}, the gain on the regulating term $(E_i-E_i^*)$ in \eqref{Eq:VoltDroopQuadratic} now scales with the  inverter voltage. 
%
%
%
%
%
Combining the power flow \eqref{Eq:RPFEDecoupled} with a static load model $Q_i(E_i)$ at each load bus $i \in \mathcal{L}$, we must also satisfy the power balance equations
\begin{equation}\label{Eq:LoadBuses}
0 = Q_i(E_i) - Q_{\mathrm{e},i}(E)\,,\quad i \in \mathcal{L}\,.
\end{equation}
Combining now the inverter model \eqref{Eq:InverterModel}, the quadratic droop controller \eqref{Eq:VoltDroopQuadratic}, the load power balance \eqref{Eq:LoadBuses}, and the power flow equation \eqref{Eq:QVec}, the closed-loop dynamical system is differential-algebraic, and can be written compactly as
\begin{align}\label{Eq:VoltDroopQuadraticClosed}
\begin{pmatrix} \vzeros[n] \\ \tau_I\dot{E}_I\end{pmatrix} &= \begin{pmatrix}Q_L(E_L) \\ [E_I]K_I(E_I-E_I^*)\end{pmatrix} + [E]BE\,,
\end{align}
where $E_I = (E_{n+1},\ldots,E_{n+m})$, $Q_L(E_L) = (Q_1(E_1),\ldots,Q_n(E_n))$, and $\tau_I$ and $K_I$ are diagonal matrices with elements $\tau_i$ and $K_i$ respectively.
%
%
Since the variables $E_i$ represent voltage magnitudes referenced to ground, they are intrinsically positive, and for physical consistency we restrict our attention to positive voltage magnitudes. 


\smallskip

\begin{remark}\textbf{(Interpretations)}\label{Rem:Interpretations}
Before proceeding to our main analysis, we offer two interpretations of the quadratic controller \eqref{Eq:VoltDroopQuadratic}, one theoretical and one pragmatic.

\emph{Circuit-theoretic interpretation:} 
The design \eqref{Eq:VoltDroopQuadratic} can be interpreted as \emph{control by interconnection}, where we interconnect the physical electrical network with fictitious ``controller circuits'' at the inverter buses $i \in \mathcal{I}$ \cite{AJvdS:10}.
\begin{figure}[h]
\begin{center}
\includegraphics[width=0.33\columnwidth]{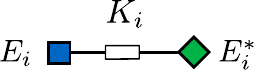}
\caption[Linear circuit representation of the quadratic droop controller \eqref{Eq:VoltDroopQuadratic}.]{Linear circuit representation of the quadratic droop controller \eqref{Eq:VoltDroopQuadratic}.}
\label{Fig:ControllerCircuit}
\end{center}
\end{figure}
\begin{figure*}[thbp]
\begin{center}
\includegraphics[width=2\columnwidth]{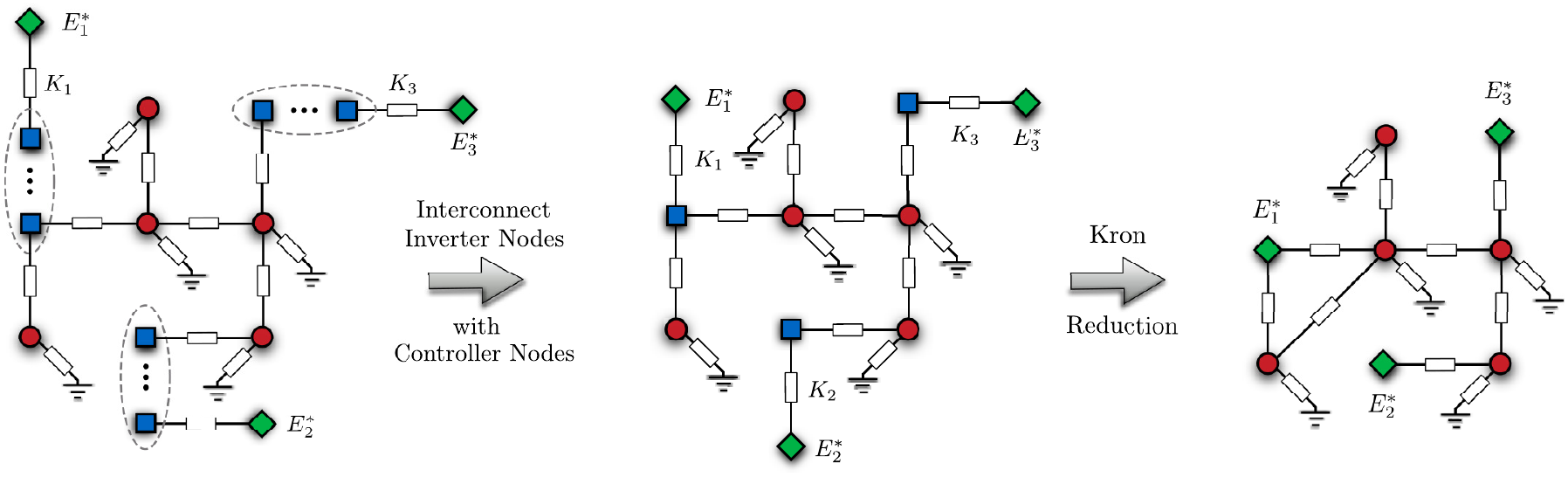}
\caption[Diagram showing network augmentation and reduction. First, each inverter bus of the network in Figure \ref{Fig:Circuit1} is interconnected with a two-bus controller circuit, consisting of an inverter bus and fictitious bus at constant voltage $E_i^*$. In Theorem \ref{Thm:ReducedFixedPoint}, the inverter buses are eliminated via Kron reduction, leaving a reduced network with only fixed voltage buses and load buses.]{Diagram showing network augmentation and reduction. First, each inverter bus {\tikz{\path[draw=black,semithick,fill={rgb:red,0;green,102;blue,196}] (0,0) rectangle (0.12cm,0.12cm);}} of the network in Figure \ref{Fig:Circuit1} is interconnected with a two-bus controller circuit, consisting of an inverter bus {\tikz{\path[draw=black,semithick,fill={rgb:red,0;green,102;blue,196}] (0,0) rectangle (0.12cm,0.12cm);}} and fictitious bus {\tikz \node[shape=diamond,draw=black,fill={rgb:red,0;green,179;blue,70},scale=0.5,] {};} at constant voltage $E_i^*$. In Theorem \ref{Thm:ReducedFixedPoint}, the inverter buses are eliminated via Kron reduction, leaving a reduced network with only fixed voltage buses {\tikz \node[shape=diamond,draw=black,fill={rgb:red,0;green,179;blue,70},scale=0.5,] {};} and load buses {\tikz\draw[black,semithick,fill=red] (0,0) circle (.5ex);}.}
\label{Fig:Circuit2}
\end{center}
\end{figure*}
Indeed, consider the two-bus circuit of Figure \ref{Fig:ControllerCircuit}, where the blue bus has variable voltage $E_i$ and is connected via a susceptance $K_i < 0$ to the green bus of fixed voltage $E_i^*$. The current-voltage relations for this fictitious circuit are
\begin{equation}\label{Eq:ControllerCircuit}
\begin{pmatrix}I_i \\ I_i^*\end{pmatrix} = \boldsymbol{\mathrm{j}}\begin{pmatrix}K_i & -K_i \\ -K_i & K_i\end{pmatrix}\begin{pmatrix}E_i \\ E_i^*\end{pmatrix}\,,
\end{equation}
where $I_i$ (resp. $I_i^*$) is the current injection at the bus with voltage $E_i$ (resp. $E_i^*$). Now, let there be $m$ of the two-bus circuits in Figure \ref{Fig:ControllerCircuit}; one for each inverter. If we identify the variable-voltage blue buses of these circuits with the inverter buses of our original network, and impose that the current injected into the former must exit from the latter, we obtain an augmented network with $n + 2m$ buses, and in vector notation the current-voltage relations in the new network are
\begin{equation}\label{Eq:AugNetwork}
\begin{pmatrix}I_L \\ \vzeros[m] \\ I_I^*\end{pmatrix} = \boldsymbol{\mathrm{j}}\begin{pmatrix}B_{LL} & B_{LI} & \vzeros[]\\ B_{IL} & B_{II} + K_I & -K_I \\ \vzeros[] & -K_I & K_I\end{pmatrix}\begin{pmatrix}E_L \\ E_I \\ E_I^*\end{pmatrix}\,.
\end{equation}
where we have block partitioned all variables according to loads, inverters, and fictitious controller buses. In this augmented circuit, the \emph{inverters behave as interior nodes} joining the fictitious controller buses to the loads, and do not sink or source power themselves. Left multiplying the first two blocks of equations in \eqref{Eq:AugNetwork} by $[E]$ and noting that by the definition of reactive power $\boldsymbol{\mathrm{j}}[E_L]I_L = Q_L(E_L)$, we immediately obtain the right hand side of \eqref{Eq:VoltDroopQuadraticClosed}. {\tb Thus, the closed-loop equilibrium points under quadratic droop control can be interpreted as the solutions of the power balance equations for an expanded linear circuit.}


\emph{Practical interpretation: } Far from being linear, voltage/reactive power capability characteristics of synchronous generators display significant nonlinearities. {\tb In the absence of saturation constraints, the characteristics are in fact quadratic  \cite[Equation 3.105]{JM-JWB-JRB:08}, and thus the quadratic droop controller \eqref{Eq:VoltDroopQuadratic} more accurately mimics the behavior of a synchronous generator with automatic voltage regulation, compared to the classical controller \eqref{Eq:DroopReactive}.} {\tb This quadratic dependence means that the marginal voltage drop (voltage drop per unit increase in reactive power) increases with reactive power provided.}
\oprocend
\end{remark}

\begin{remark}\textbf{(Generalizations)}\label{Rem:Generalizations}
The quadratic droop controller \eqref{Eq:VoltDroopQuadratic} is a special case of the general feedback controller
\begin{equation}\label{Eq:GenQuad}
u_i = E_i \sum_{j\in\mc I}\nolimits \left(\alpha_{ij}E_j + \beta_{ij}E_j^*\right)\,,
\end{equation}
where $\alpha_{ij}=\alpha_{ji}$ and $\beta_{ij}$ are gains. One recovers \eqref{Eq:VoltDroopQuadratic} by setting $\alpha_{ii} = K_i$, $\beta_{ii} = -K_i$, and all other parameters to zero. While the decentralized controller \eqref{Eq:VoltDroopQuadratic} can be interpreted as control-by-interconnection with the circuit of Figure \ref{Fig:ControllerCircuit}, the controller \eqref{Eq:GenQuad} represents a more general, densely interconnected circuit with $m$ variable-voltage nodes and $m$ fixed voltage nodes. Since decentralized control strategies are preferable in microgrids, we focus on the decentralized controller \eqref{Eq:VoltDroopQuadratic} with the understanding that results may be extended to the more general \eqref{Eq:GenQuad}. 
\hfill \oprocend
\end{remark}

\subsection{Equilibria and Stability Analysis by Network Reduction}
\label{Subsec: Stability of Closed-Loop System}



%
%
%
%


\smallskip



%
%
%
%


We first pursue the following question: under what conditions on load, network topology, admittances, and controller gains does the differential-algebraic closed-loop system \eqref{Eq:VoltDroopQuadraticClosed} possess a locally exponentially stable equilibrium? By exploiting the structure of the quadratic droop controller \eqref{Eq:VoltDroopQuadratic}, we will establish a correspondence between the equilibria of \eqref{Eq:VoltDroopQuadraticClosed} and the solutions of a power flow equation for a \emph{reduced} network. In this subsection we focus on generic static load models $Q_i(E_i)$ before addressing  specific load models in Section \ref{Subsec:LoadModels}.

For notational convenience, we first define a few useful quantities. We block partition the susceptance matrix and nodal voltage variables according to loads and inverters as 
\begin{equation}\label{Eq:BPartition}
B = \begin{pmatrix}B_{LL} & B_{LI} \\ B_{IL} & B_{II}\end{pmatrix}\,,\qquad E = (E_L,E_I)\,,
\end{equation}
and define the \emph{reduced susceptance matrix} $B_{\rm red} \in \real^{n \times n}$ by
\begin{equation}\label{Eq:ReducedNetworkGeneral}
B_{\rm red} \triangleq  B_{LL}-B_{LI}\left(B_{II}+K_I\right)^{-1}B_{IL}\,.
\end{equation}
Moreover, we define the \emph{averaging matrices} $W_1 \in \real^{n \times m}$ and $W_2 \in \real^{m \times (n+m)}$ by
\begin{align}
W_1 &\triangleq  -B_{\rm red}^{-1}B_{LI}\left(B_{II}+K_I\right)^{-1}K_I\,,
\label{Eq:A1}\\
W_2 &\triangleq (B_{II}+K_I)^{-1}\begin{pmatrix}-B_{IL} & K_I\end{pmatrix}\,.
\label{Eq:A2}
\end{align}
It can be shown (Proposition  \ref{Prop:AveragingMatrices}) that $B_{\rm red}$ is invertible, and that $W_1$ and $W_2$ are both row-stochastic matrices. Finally we define the \emph{open-circuit load voltages} $E_L^* \in \real^{n}_{>0}$ via
\begin{equation}\label{Eq:Eaverage}
E_L^{*} \triangleq W_1E_I^*\,.
\end{equation}
Since $W_1$ is row-stochastic, each component of $E_L^*$ is a weighted average of inverter set points $E_i^*$.

\smallskip
\begin{theorem}\textbf{(Reduced Power Flow Equation for Quadratic Droop Network)}\label{Thm:ReducedFixedPoint}
Consider the closed-loop system \eqref{Eq:VoltDroopQuadraticClosed} resulting from the quadratic droop controller \eqref{Eq:VoltDroopQuadratic}, along with the definitions \eqref{Eq:ReducedNetworkGeneral}--\eqref{Eq:Eaverage}.
%
%
The following two statements are equivalent:
\begin{enumerate}
\item[(i)] \textbf{Original Network: }The voltage vector $E = (E_L,E_I) \in \real^{n+m}_{>0}$ is an equilibrium point of \eqref{Eq:VoltDroopQuadraticClosed};
\item[(ii)] \textbf{Reduced Network: }The load voltage vector $E_L \in \real^n_{>0}$ is a solution of the reduced power flow equation
\begin{equation}\label{Eq:ReducedFixedPoint}
\vzeros[n] = Q_L(E_L) + [E_L]B_{\rm red}(E_L - E_L^*)\,,
\end{equation}
and the inverter voltage vector is recovered via
\begin{align}\label{Eq:InverterVoltages}
E_I &= W_2\begin{pmatrix}E_L \\ E_I^*\end{pmatrix} \in \real^{m}_{>0}\,.
\end{align}
\end{enumerate}
\end{theorem}\medskip

Theorem \ref{Thm:ReducedFixedPoint} states that, to study the existence and uniqueness of equilibria for the closed-loop system \eqref{Eq:VoltDroopQuadraticClosed}, we need only study the reduced power flow equation \eqref{Eq:ReducedFixedPoint}. Voltages at inverters may then be recovered uniquely from \eqref{Eq:InverterVoltages}. In fact, since $W_2$ is row-stochastic (Proposition \ref{Prop:AveragingMatrices}), the inverter voltages \eqref{Eq:InverterVoltages} are simply weighted averages of load voltages $E_L$ and inverter set point voltages $E_I^*$. 

The reduced power flow equation \eqref{Eq:ReducedFixedPoint} is straightforward to interpret in terms of the circuit-reduction of Figure \ref{Fig:Circuit2}. First, as in Remark \ref{Rem:Interpretations}, we augment the original circuit with fictitious controller buses. Eliminating the inverter voltages $E_I$ from the augmented network current balance \eqref{Eq:AugNetwork} through Kron reduction \cite{fd-fb:11d}, one obtains the input/output equivalent circuit
\begin{equation}\label{Eq:Interp1}
\begin{pmatrix}I_L \\ I_I^*\end{pmatrix} = \boldsymbol{\mathrm{j}}\begin{pmatrix}B_{\rm red} & -B_{\rm red}W_1 \\ -W_1^{\sf T}B_{\rm red} & K_I(B_{II}+K_I)^{-1}B_{II}\end{pmatrix}\begin{pmatrix}E_L \\ E_I^*\end{pmatrix}\,.
\end{equation}
This reduction process is shown in Figure \ref{Fig:Circuit2}. Left-multiplying the first block in \eqref{Eq:Interp1} by $[E_L]$ immediately yields the reduced power flow \eqref{Eq:ReducedFixedPoint}. Hence \eqref{Eq:ReducedFixedPoint} is exactly the reactive power balance equation $Q_L = [E_L]I_L$ in the reduced network of Figure \ref{Fig:Circuit2} after reduction (cf. \cite[Equation 2.10b]{TVC-CV:98}). 
The bottom block of equations in \eqref{Eq:Interp1} determines the fictitious controller current injections once the top block is solved for $E_L$.


\begin{pfof}{Theorem \ref{Thm:ReducedFixedPoint}}
\textbf{(i)}$\Rightarrow$\textbf{(ii): } 
Setting the left-hand side of the closed-loop system \eqref{Eq:VoltDroopQuadraticClosed} to zero, equilibrium points satisfy
\begin{align}\label{Eq:NonReducedFixedPoint}
\vzeros[n+m] &= \begin{pmatrix}Q_L(E_L) \\ [E_I]K_I(E_I-E_I^*)\end{pmatrix} \\\nonumber
&\quad\,\,+ \begin{pmatrix}[E_L] & \vzeros[] \\ \vzeros[] & [E_I]\end{pmatrix}\begin{pmatrix}B_{LL} & B_{LI} \\ B_{IL} & B_{II}\end{pmatrix}\begin{pmatrix}E_L \\ E_I\end{pmatrix}.
\end{align}
Since $E \in \real^{n+m}_{>0}$ solves \eqref{Eq:NonReducedFixedPoint}, we can left-multiply the lower block of equations in \eqref{Eq:NonReducedFixedPoint} by $[E_I]^{-1}$ to obtain
\begin{equation}\label{Eq:ZeroI}
\vzeros[m] = K_I(E_I-E_I^*) + B_{II}E_I + B_{IL}E_L\,,
\end{equation}
and solve for $E_I$ to obtain the inverter voltages \eqref{Eq:InverterVoltages}. Substituting \eqref{Eq:InverterVoltages} into the first block of equations in \eqref{Eq:NonReducedFixedPoint}, we calculate
\begin{align*}
-Q_L(E_L) &=  [E_L]B_{LL}E_L\\ &\quad- [E_L]B_{LI}(B_{II}+K_I)^{-1}\left(B_{IL}E_L-K_IE_I^*\right)\\
&= [E_L]\underbrace{(B_{LL}-B_{LI}(B_{II}+K_I)^{-1}B_{IL})}_{B_{\rm red}}E_L\\
&\quad - [E_L]\underbrace{B_{LI}(B_{II}+K_I)^{-1}K_IE_I^*}_{B_{\rm red}E_L^*}\\
&= [E_L]B_{\rm red}(E_L-E_L^*)\,,
\end{align*}
which is the reduced power flow equation \eqref{Eq:ReducedFixedPoint}.\newline
\textbf{(ii)}$\Rightarrow$\textbf{(i): } Due to \eqref{Eq:InverterVoltages} and Proposition \ref{Prop:AveragingMatrices} (iii), 
we have that $E_L \in \real^{n}_{>0}$ implies that $E_I \in \real^{m}_{>0}$ and hence $E \in \real^n_{> 0}$. An easy computation shows that \eqref{Eq:ReducedFixedPoint} and \eqref{Eq:InverterVoltages} together imply that $(E_L,E_I)$ satisfy the fixed-point equations
\eqref{Eq:NonReducedFixedPoint}, and thus $E \in \real^{n+m}_{>0}$ is an equilibrium point. 
\end{pfof}

In a similar spirit, the following result states that local exponential stability of an equilibrium point of \eqref{Eq:VoltDroopQuadraticClosed} may be checked by studying a \emph{reduced} Jacobian matrix.

\begin{theorem}\textbf{(Stability from Reduced Jacobian)}\label{Thm:Stab}
Consider the closed-loop system \eqref{Eq:VoltDroopQuadraticClosed} resulting from the quadratic droop controller \eqref{Eq:VoltDroopQuadratic}. If the Jacobian of the reduced power flow equation \eqref{Eq:ReducedFixedPoint}, given in vector notation by
\begin{equation}\label{Eq:ThmJac}
J_{\rm red}(E_L) = \frac{\partial Q_L}{\partial E_L}(E_L) + [E_L]B_{\rm red} + [B_{\rm red}(E_L-E_L^*)]\,,
\end{equation}
is a Hurwitz matrix when evaluated at a solution $E_L \in \real^{m}_{>0}$ of \eqref{Eq:ReducedFixedPoint}, then the equilibrium point $(E_L,E_I)$ of the differential-algebraic system \eqref{Eq:VoltDroopQuadraticClosed} is locally exponentially stable. {\tb Moreover, assuming that $\frac{\partial Q_i(E_i)}{\partial E_i} \leq 0$ for each load bus $i \in \mathcal{L}$, a sufficient condition for \eqref{Eq:ThmJac} to be Hurwitz is that
\begin{equation}\label{Eq:SuffImplicitStability}
B_{\rm red} \prec [E_L]^{-2}[Q_L(E_L)]\,.
\end{equation}
}
\end{theorem}

\begin{pfof}{Theorem \ref{Thm:Stab}} We appeal to \cite[Theorem 1]{RR:04}, which states that local stability of the differential-algebraic system \eqref{Eq:VoltDroopQuadraticClosed} at the equilibrium $E \in \real^{n+m}_{>0}$ may be studied by linearizing the differential algebraic system, eliminating the algebraic equations from the system matrix, and checking that the resulting reduced matrix is Hurwitz. Consider the generalized eigenvalue problem (GEP) $Jv = \lambda \tau v$ where $v \in \real^{n}$, $\tau = \mathrm{blkdiag}(\vzeros[],\tau_I)$ and Jacobian matrix $J$ of \eqref{Eq:VoltDroopQuadraticClosed} is given by
\begin{equation}\label{Eq:QuadraticDroopJacobian}
J = [E]B + [BE] + D\,.
\end{equation}
The diagonal matrix $D \in \real^{(n+m)\times (n+m)}$ has elements $D_{ii} = \mathrm{d}Q_i(E_i)/\mathrm{d}E_i$
for $i \in \mathcal{L}$ and $D_{ii} = K_{i}(E_{i}-E_{i}^*) + K_iE_i$ for $i \in \mathcal{I}$. Since $E$ is strictly positive, we left-multiply through by $[E]^{-1}$ and formulate the previous GEP as the symmetric GEP
\begin{equation}
	Mv = \lambda[E]^{-1}\tau v\,,
	\label{eq: generalized EV problem}
\end{equation}
where 
\begin{align}\label{Eq:MMatrix}
M = M^{\sf T} \triangleq B + [E]^{-1}([BE]+D)\,.
\end{align}
Partitioning the eigenvector $v$ as $v = (v_L,v_I)$, block partitioning $M$, and eliminating the top set of algebraic equations, we arrive at a reduced GEP $M_{\rm red}v_I = \lambda[E_I]^{-1}\tau_I v_I$ where $M_{\rm red} = M_{\rm red}^{\sf T} \triangleq M_{II}-M_{IL}M_{LL}^{-1}M_{LI}$. 
%
%
%
%
%
The matrices on both sides are symmetric, and in particular the matrix $[E_I]^{-1}\tau_I$ on the right is diagonal and positive definite. The eigenvalues $\{\lambda_i\}_{i\in \mathcal{I}}$ of this reduced GEP are therefore real, and it holds that $\lambda_i < 0$ for each $i \in \mathcal{I}$ if and only if $M_{\rm red}$ is negative definite.
We now show indirectly that $-M_{\rm red}$ is positive definite, by combining two standard results on Schur complements. Through some straightforward computations, one may use \eqref{Eq:ReducedNetworkGeneral}, \eqref{Eq:Eaverage}, \eqref{Eq:InverterVoltages} and \eqref{Eq:ZeroI} to simplify $M$ to
$$
M = \begin{pmatrix}M_{11} & B_{LI} \\ B_{IL} & B_{II}+ K_I\end{pmatrix}\,,
$$
where $M_{11} = [E_L]^{-1}(\partial Q_L/\partial E_L) + B_{LL} + [E_L]^{-1}[B_{\rm red}(E_L-E_L^*)]$. Since the bottom-right block $-(B_{II}+K_I)$ of $-M$ is positive definite, $-M$ will be positive definite if and only if the Schur complement with respect to this bottom-right block is also positive definite. This Schur complement is equal to 
\begin{equation}\label{Eq:Schur1}
-[E_L]^{-1}\frac{\partial Q_L}{\partial E_L}(E_L) - B_{\rm red} - [E_L]^{-1}[B_{\rm red}(E_L-E_L^*)]\,,
\end{equation}
which is exactly $-[E_L]^{-1}$ times the Jacobian \eqref{Eq:ThmJac} of \eqref{Eq:ReducedFixedPoint}. Since $J_{\rm red}$ is Hurwitz with nonnegative off-diagonal elements, $-J$ is an $M$-matrix and is therefore $D$-stable. It follows that the Schur complement \eqref{Eq:Schur1} is positive definite, and hence that $M_{\rm red}$ is positive definite, which completes the proof of the main statement. 
{\tb For the moreover statement, proceed along similar arguments and consider the symmetric version $[E_L]^{-1}J_{\rm red}(E_L)$ of \eqref{Eq:ThmJac}, which equals
$$
[E_L]^{-1}\frac{\partial Q_L}{\partial E_L}(E_L) + B_{\rm red} + [E_L]^{-1}[B_{\rm red}(E_L-E_L^*)]\,.
$$
Solving the reduced power flow equation \eqref{Eq:ReducedFixedPoint} for $B_{\rm red}(E_L-E_L^*)$ and substituting into the third term above, we find that
$$
[E_L]^{-1}J_{\rm red} = [E_L]^{-1}\frac{\partial Q_L}{\partial E_L}(E_L) + B_{\rm red} - [E_L]^{-2}[Q_L(E_L)]\,.
$$
The first term is diagonal and by assumption negative semidefinite, and hence $J_{\rm red}(E_L)$ will be Hurwitz if \eqref{Eq:SuffImplicitStability} holds, which completes the proof.
}
\end{pfof}

\smallskip

{\tb The sufficient condition \eqref{Eq:SuffImplicitStability} is intuitive: it states that at equilibrium, the network matrix $B_{\rm red}$ should be more susceptive than the equivalent load susceptances $Q_i/E_i^2$. The results of Theorem \ref{Thm:Stab} are implicit}, in that checking stability depends on the undetermined solutions of the reduced power flow equation \eqref{Eq:ReducedFixedPoint}. The situation will become clearer in Section \ref{Subsec:LoadModels} when we apply Theorems \ref{Thm:ReducedFixedPoint} and \ref{Thm:Stab} to specific load models.

\smallskip

\begin{remark}\textbf{(Parallel Microgrids)}\label{Rem:Parallel}
In \cite{jwsp-fd-fb:13h} we provided an extensive analysis of the closed loop \eqref{Eq:VoltDroopQuadraticClosed} for a ``parallel'' or ``star''  topology, where all inverters feed  a single common load. In particular, for a constant power load we provided a necessary and sufficient condition for the existence of a stable equilibrium. Depending on the system parameters, the microgrid can have zero, one, or two physically meaningful equilibria, and displays both saddle-node and singularity-induced bifurcations. Perhaps surprisingly, we show that even for a star topology, overall network stability is \emph{not} equivalent to pair-wise stability of each inverter with the common load.
\hfill \oprocend
\end{remark}

%

\smallskip

\subsection{Equilibria and Stability Conditions for ZI, ZIP, and Dynamic Shunt Load Models}
\label{Subsec:LoadModels}

In this section we leverage the general results of Section \ref{Subsec: Stability of Closed-Loop System} to study the equilibria and stability of the microgrid for some specific, tractable load models. To begin, the reduced power flow equation \eqref{Eq:ReducedFixedPoint} can be solved \emph{exactly} for combinations of constant-impedance and constant-current loads.

\begin{theorem}\textbf{(Stability with ``ZI'' Loads)}\label{Lem:ZILoads}
Consider the reduced power flow equation \eqref{Eq:ReducedFixedPoint} for constant-impedance/constant-current loads
\begin{equation}\label{Eq:ZILoad}
Q_L^{\rm ZI}(E_L) = [E_L][b_{\rm shunt}]E_L + [E_L]I_{\rm shunt}\,,
\end{equation}
where $b_{\rm shunt} \in \real^n$ (resp. $I_{\rm shunt } \in \real^{n}$) is the vector of constant-impedance loads (resp. constant current loads). Assume that %
\begin{enumerate}
\item[(i)] $-(B_{\rm red}+[b_{\rm shunt}])$ is an $M$-matrix\,, and
\item[(ii)] $I_{\rm shunt} > B_{\rm red}E_L^*$ component-wise\,.
\end{enumerate}
Then the unique solution $E_L^{\rm ZI} \in \real_{>0}^{n}$ to \eqref{Eq:ReducedFixedPoint} is given by
\begin{align}\label{Eq:ExactSolutionZILoad}
E_L^{\rm ZI} &= (B_{\rm red}+[b_{\rm shunt}])^{-1}(B_{\rm red}E_L^*-I_{\rm shunt})\,,
\end{align}
and the associated equilibrium point $(E_L^{\rm ZI},E_I^{\rm ZI})$ of the closed-loop system \eqref{Eq:VoltDroopQuadraticClosed} is locally exponentially stable.
\end{theorem}
\smallskip

The first technical condition in Theorem \ref{Lem:ZILoads} restricts the impedance loads from being overly capacitive, while the second restricts the current loads from being overly inductive (since $B_{\rm red}E_L^* < \vzeros[n]$ component-wise). Both would be violated only under fault conditions in the microgrid, and are therefore non-conservative.
As expected, in the case of open-circuit operation (no loading) when $I_{\rm shunt} = b_{\rm shunt} = \vzeros[n]$, \eqref{Eq:ExactSolutionZILoad} reduces to $E_L^{\rm ZI} = E_L^*$, the open-circuit load voltage vector.

\smallskip

\begin{pfof}{Theorem \ref{Lem:ZILoads}}\,
Substituting the ZI load model \eqref{Eq:ZILoad} into the reduced power flow equation \eqref{Eq:ReducedFixedPoint}, we obtain
\begin{align}\nonumber
\vzeros[n] &= [E_L]\left((B_{\rm red}+[b_{\rm shunt}])E_L - (B_{\rm red}E_L^*-I_{\rm shunt})\right)\\
\label{Eq:ZIEquation}
&= [E_L](B_{\rm red}+[b_{\rm shunt}])(E_L-E_L^{\rm ZI})\,,
\end{align}
where in the second line we have factored out $(B_{\rm red}+[b_{\rm shunt}])$ and identified the second term in parentheses with \eqref{Eq:ExactSolutionZILoad}. Observe that $E_L = E_L^{\rm ZI}$ is a solution to \eqref{Eq:ZIEquation}. Since $(B_{\rm red}+[b_{\rm shunt}])$ is nonsingular and we require that $E_L \in \real^n_{>0}$, $E_L^{\rm ZI}$ is the \emph{unique} solution of \eqref{Eq:ZIEquation}. It remains only to show that $E_L^{\rm ZI} \in \real^n_{>0}$. Since $-(B_{\rm red} + [b_{\rm shunt}])$ is by assumption an $M$-matrix, the inverse $(B_{\rm red}+b_{\rm [shunt}])^{-1}$ has nonpositive elements \cite{AB-RJP:94}. By invertibility, $(B_{\rm red}+b_{\rm [shunt}])^{-1}$ cannot have any zero rows, and therefore the product $(B_{\rm red}+b_{\rm [shunt}])^{-1}(B_{\rm red}E_L^*-I_{\rm shunt})$ with the strictly negative vector $B_{\rm red}E_L^*-I_{\rm shunt}$ yields a strictly positive vector $E_L^{\rm ZI} \in \real_{>0}^{n}$ as in \eqref{Eq:ExactSolutionZILoad}.

To show local stability, we apply Theorem \ref{Thm:Stab} and check that the Jacobian of \eqref{Eq:ZIEquation} evaluated at \eqref{Eq:ExactSolutionZILoad} is Hurwitz. Differentiating \eqref{Eq:ZIEquation}, we calculate that
$$
J_{\rm red} = [E_L](B_{\rm red}+[b_{\rm shunt}]) + [(B_{\rm red}+[b_{\rm shunt}])(E_L-E_L^{\rm ZI})]\,,
$$
and therefore that
\begin{equation}\label{Eq:JredELZI}
J_{\rm red}(E_L^{\rm ZI}) = [E_L^{\rm ZI}](B_{\rm red}+[b_{\rm shunt}])\,,
\end{equation}
which is Hurwitz since $-(B_{\rm red}+[b_{\rm shunt}])$ is by assumption an $M$-matrix and $M$-matrices are $D$-stable.
\end{pfof}

\smallskip

A more general static load model still is the ZIP model $Q_i(E_i) = b_{\mathrm{shunt},i}E_i^2 + I_{\mathrm{shunt},i}E_i + Q_i$ which augments the ZI model \eqref{Eq:ZILoad} with an additional constant power demand $Q_i \in \real$ \cite{TVC-CV:98}. The model is depicted in Figure \ref{Fig:ZIP}.
\begin{figure}[th]
\begin{center}
\includegraphics[width=0.8\columnwidth]{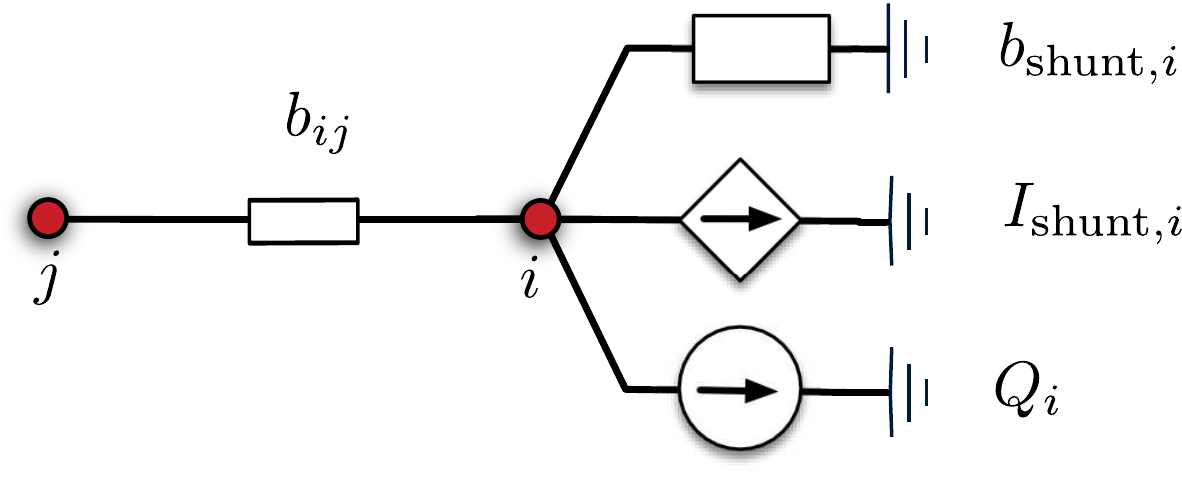}
\caption[]{Pictorial representation of ZIP load model.}
\label{Fig:ZIP}
\end{center}
\end{figure}
In vector notation, the ZIP model generalizes \eqref{Eq:ZILoad} as
\begin{equation}\label{Eq:ZIP}
Q_L^{\rm ZIP}(E_L) = [E_L][b_{\rm shunt}]E_L + [E_L]I_{\rm shunt} + Q_L\,.
\end{equation}
Unlike the ZI model \eqref{Eq:ZILoad}, the reduced power flow equation \eqref{Eq:ReducedFixedPoint} cannot be solved analytically for ZIP loads, except for the special case of a parallel microgrid. Even in this special case, the network can have multiple equilibrium points with non-trivial stability properties \cite{jwsp-fd-fb:13h}. The following result builds on analysis techniques developed in \cite{BG-JWSP-FD-SZ-FB:13zb} and provides an approximate characterization of the high-voltage solution to \eqref{Eq:ReducedFixedPoint} {\tb when the constant power term $Q_L$ is ``small''.}

\begin{theorem}\textbf{(Stability with ``ZIP'' Loads)}\label{Thm:ZIP} Consider the reduced power flow equation \eqref{Eq:ReducedFixedPoint} with the ZIP load model \eqref{Eq:ZIP}, let the conditions of Theorem \ref{Lem:ZILoads} hold, and let $E_L^{\rm ZI}$ be the high-voltage solution of \eqref{Eq:ReducedFixedPoint} for ZI loads as given in Theorem \ref{Lem:ZILoads}. Furthermore, define the \emph{short-circuit capacity matrix}
\begin{equation}\label{Eq:QZI}
Q_{\rm sc} \triangleq [E_L^{\rm ZI}](B_{\rm red}+[b_{\rm shunt}])[E_L^{\rm ZI}] \in \real^{n\times n}\,.
\end{equation}
If $\|Q_L\|$ is sufficiently small, then there exists a unique high-voltage solution $E_L^{\rm ZIP} \in \real^{n}_{>0}$ of the reduced power flow equation \eqref{Eq:ReducedFixedPoint} with ZIP loads \eqref{Eq:ZIP}, given by
\begin{equation}\label{Eq:ELapprox}
E_L^{\rm ZIP} = [E_L^{\rm ZI}]\left(\vones[n] - Q_{\rm sc}^{-1}Q_L + \epsilon\right)\,,
\end{equation}%
where $\|\epsilon\| = \mathcal{O}(\|Q_{\rm sc}^{-1}Q_L\|^2)$, and the corresponding equilibrium point $(E_L^{\rm ZIP},E_I^{\rm ZIP})$ of \eqref{Eq:VoltDroopQuadraticClosed} is locally exponentially stable.
\end{theorem}

\begin{proof}\,See Appendix A.
\end{proof}

\smallskip

Similar to the technical conditions (i) and (ii) in Theorem \ref{Lem:ZILoads} regarding constant-impedance and constant-current load components, Theorem \ref{Thm:ZIP} requires the constant-power component $Q_L$ of the load model to be sufficiently small. The high-voltage solution \eqref{Eq:ELapprox} can be thought of as a regular perturbation of the solution $E_L^{\rm ZI}$ for ZI loads, where $\epsilon$ is a small error term. As $\|Q_{\rm sc}^{-1}Q_L\| \rightarrow 0$, $\|\epsilon\| \rightarrow 0$ quadratically, and $E_L^{\rm ZIP} \rightarrow E_L^{\rm ZI}$.

\smallskip

\smallskip

While the ZIP load model can accurately capture the steady-state behavior of most aggregated loads, when considering dynamic stability of a power system it is often important to check results obtained for static load models against those obtained using basic dynamic load models \cite{TVC-CV:98}. A common dynamic load model is the \emph{dynamic shunt} model $T_i\dot{b}_{\mathrm{dyn-shunt},i} = Q_i - E_i^2 b_{\mathrm{dyn-shunt},i}$, or in vector notation
\begin{equation}\label{Eq:DynamicShunt}
T\dot{b}_{\mathrm{dyn-shunt}} = Q_L - [E_L]^2 b_{\mathrm{dyn-shunt}}\,.
\end{equation}
where $T$ is a diagonal matrix of time constants. The model specifies a constant-impedance load model $Q_i(E_i) = b_{\mathrm{dyn-shunt},i}E_i^2$, with the shunt susceptance $b_{\mathrm{dyn-shunt},i}$ dynamically adjusted to achieve a constant power consumption $Q_i$ in steady-state. This is a common low-fidelity dynamic model for thermostatically controlled loads, induction motors, and loads behind tap-changing transformers \cite{TVC-CV:98}. We restrict our attention to inductive loads $Q_i \leq 0$, as these are the most common in practice, and without loss of generality we assume that $Q_i < 0$ for all $i \in \mathcal{L}$, since if $Q_i = 0$ the unique steady state of \eqref{Eq:DynamicShunt} is $b_{\mathrm{dyn-shunt},i} = 0$ and the equation can be removed.

\smallskip

\begin{theorem}\textbf{(Stability with Dynamic Shunt Loads)}\label{Thm:StabDS} Consider the reduced power flow equation \eqref{Eq:ReducedFixedPoint} with the dynamic shunt load model \eqref{Eq:DynamicShunt} and $Q_{i} < 0$. If $\|Q_L\|$ is sufficiently small, then there exists a unique solution $(E_L^{\rm DS},b_{\rm dyn-shunt}^{\rm DS}) \in \real^{n}_{>0} \times \real_{<0}^n$ of \eqref{Eq:ReducedFixedPoint},\eqref{Eq:DynamicShunt}, given by
\begin{subequations}
\begin{align}\label{Eq:EDS}
E_{L}^{\rm DS} &= [E_L^*]\left(\vones[n]-Q_{\rm sc}^{-1}Q_L+\epsilon\right)\,,\\
\label{Eq:bDS}
b_{\rm dyn-shunt}^{\rm DS} &= [E_L^{\rm DS}]^{-2}Q_L\,,
\end{align}
\end{subequations}
where $Q_{\rm sc} = [E_L^*]B_{\rm red}[E_L^*]$ and $\|\epsilon\| = \mathcal{O}(\|Q_{\rm sc}^{-1}Q_L\|^2)$, and the corresponding equilibrium point $(E_L^{\rm DS},E_I^{\rm DS},b_{\rm dyn-shunt}^{\rm DS})$ of the system \eqref{Eq:VoltDroopQuadraticClosed},\eqref{Eq:DynamicShunt} is locally exponentially stable.

\end{theorem}


\begin{proof}\, We sketch the proof in Appendix A.
\end{proof}


\begin{remark}\textbf{(From Quadratic to Conventional Droop):}\label{Rem:FromQuad}
While the stability results derived above hold for the quadratic droop-controlled microgrid, under certain selections of controller gains a direct implication can be drawn from the above stability results to stability of the microgrid under conventional, linear voltage droop control. Additional information can be found in  \cite[Section V]{jwsp-fd-fb:13h}. \oprocend
\end{remark}

\smallskip

We conclude this section by commenting on the utility of the preceding analysis and stability results. If hard design limits are imposed on the load voltages $E_L$, the expressions \eqref{Eq:ExactSolutionZILoad} or \eqref{Eq:ELapprox} for the unique high-voltage solutions can be used for design purposes to (non-uniquely) back-calculate controller gains, or to determine bounds on tolerable loading profiles.

\section{Controller Performance}
\label{Sec:Performance}

Having thoroughly investigated the stability properties of the closed-loop system \eqref{Eq:VoltDroopQuadraticClosed} in Section \ref{Section: Quadratic Droop Control for Voltage Stabilization}, we now turn to questions of optimality and controller performance. {\tb Here we are interested in inverse-optimality of the resulting equilibrium point of the system in an optimization sense, along with how well the resulting equilibrium achieves the power sharing objective. We do not addresses system-theoretic optimality or performance of the controller.}

\subsection{Optimality of Quadratic Droop Control}
\label{Sec:Optimal}

While in Section \ref{Section: Quadratic Droop Control for Voltage Stabilization} we emphasized the circuit-theoretic interpretation of the quadratic droop controller \eqref{Eq:VoltDroopQuadratic}, it is {\tb also relevant to ask whether the resulting equilibrium point of the system is inverse-optimal with respect to any criteria}. For simplicity of exposition, we restrict our attention in this subsection to constant-impedance load models of the form 
\begin{equation}\label{Eq:Zloads}
Q_i(E_i) = b_{\mathrm{shunt},i}E_i^2\,.
\end{equation}
Analogous extensions to ZIP load models are possible using energy function theory \cite{TJO-ID-CLM:94}, at the cost notational complexity and less insight. To begin, note that for nodal voltages $(E_L,E_I)\in\real^{n+m}_{>0}$, the total reactive power $\mathcal{Q}_{\rm loss}$ absorbed by the inductive transmission lines of the microgrid is given by 
\begin{equation*}\label{Eq:Losses}
{\mathcal{Q}}_{\rm loss}(E_L,E_I) = \sum_{\{i,j\}\in\mathcal{E}} \nolimits B_{ij}(E_i-E_j)^2\,.
\end{equation*}
%
Similarly, the total reactive power \emph{consumed} by the loads is
\begin{equation*}\label{Eq:LoadLosses}
{\mathcal{Q}}_{\rm load}(E_L) = -\sum_{i \in \mathcal{L}} \nolimits b_{\mathrm{shunt},i}E_i^2\,.
\end{equation*}
Efficient operation of the microgrid stipulates that we minimize reactive power losses and consumption, since transmission and consumption of reactive power contributes to thermal losses.\footnote{{\tb For an inverse-optimality analysis that relates real power generation to frequency droop control, see \cite{FD-JWSP-FB:13y}}.} Simultaneously however, we have the practical requirement that inverter voltages $E_i$ must remain close to their rated values $E_i^* > 0$. We encode this requirement in the quadratic cost
\begin{equation}\label{Eq:Jvolt}
C_{\rm volt}(E_I) = -\sum_{j \in \mathcal{I}}\nolimits \kappa_i(E_i-E_i^*)^2\,,
\end{equation}
where $\kappa_i < 0$ are cost coefficients. Consider now the combined optimization problem
\begin{align}
	\minimize_{E \in \real^{n + m}_{>0}} 
	& \; \; C = \mathcal{Q}_{\rm loss}(E) + \mathcal{Q}_{\rm load}(E_L) + C_{\rm volt}(E_I)\,,
	\label{Eq:OptCost}
\end{align}%
%
%
where we {\tb attempt to enforce a trade-off between minimizing reactive power dissipation and minimizing voltage deviations.} The next result shows that an appropriately designed quadratic droop controller \eqref{Eq:VoltDroopQuadratic} is a decentralized algorithm for solving the optimization problem \eqref{Eq:OptCost}. Conversely, any quadratic droop controller solves an optimization problem of the form \eqref{Eq:OptCost} for appropriate coefficients $\kappa_i$.



\smallskip

\begin{proposition}\textbf{(Optimality and Quadratic Droop)}\label{Prop:Optimization} Consider the closed-loop microgrid system \eqref{Eq:VoltDroopQuadraticClosed} with constant-impedance loads \eqref{Eq:Zloads} and controller gains $\{K_i\}_{i \in \mathcal{I}}$, and the optimization problem \eqref{Eq:OptCost} with cost coefficients $\{\kappa_i\}_{i \in \mathcal{I}}$. Assume as in Theorem \ref{Lem:ZILoads} that $-(B_{\rm red} + [b_{\rm shunt}])$ is an $M$-matrix. If the parameters are selected such that $K_i = \kappa_i$ for each $i \in \mathcal{I}$, then the unique locally exponentially stable equilibrium $(E_L^{\rm Z},E_I^{\rm Z}) \in \real^{n+m}_{>0}$ of the closed-loop system \eqref{Eq:VoltDroopQuadraticClosed} is equal to the unique minimizer of the optimization problem \eqref{Eq:OptCost}, and both are given by
\begin{align}\label{Eq:ExactSolutionZLoad}
E_L^{\rm Z} &= (B_{\rm red}+[b_{\rm shunt}])^{-1}B_{\rm red}E_L^* \in \real^n_{>0}\,,
\end{align}
with $E_I^{\rm Z} \in \real^{m}_{>0}$ recovered by inserting \eqref{Eq:ExactSolutionZLoad} into \eqref{Eq:InverterVoltages}.
\end{proposition}

\begin{pfof}{Proposition \ref{Prop:Optimization}}\,
That $(E_L^{\rm Z},E_I^{\rm Z})$ is the unique locally exponentially stable equilibrium of \eqref{Eq:VoltDroopQuadraticClosed} follows immediately from Theorem \ref{Lem:ZILoads} by setting $I_{\rm shunt} = \vzeros[n]$. We now relate the critical points of the optimization problem \eqref{Eq:OptCost} to the equilibrium equations for the closed-loop system. In vector notation, we have that $\mathcal{Q}_{\rm loss} = -E^{\sf T}BE$, $\mathcal{Q}_{\rm load} = -E_L^{\sf T}[b_{\rm shunt}]E_L$, and $C_{\rm volt} = -(E_I-E_I^*)^{\sf T} K_I(E_I-E_I^*)$, where $E = (E_1,\ldots,E_{n+m})$. The total cost $C$ may therefore be written as the quadratic form $C = x^{\sf T} \mathcal{B} x$, where $x = (E_L,E_I,E_I^*)$ and 
$$
\mathcal{B}  =
-\begin{pmatrix}
B_{LL} + [b_{\rm shunt}] & B_{LI} & \vzeros[] \\ 
B_{IL} & B_{II}+K_I & -K_I\\
\vzeros[] & -K_I & K_I
\end{pmatrix}\,.
$$
The first order optimality conditions $\nabla C(E) = \vzeros[n+m]$ yield
\begin{equation}\label{Eq:FirstOrder}
\begin{pmatrix}
B_{LL}+[b_{\rm shunt}] & B_{LI} \\ B_{IL} & B_{II}+K_I
\end{pmatrix}\begin{pmatrix}E_L \\ E_I\end{pmatrix} = \begin{pmatrix}\vzeros[n] \\ K_IE_I^*\end{pmatrix}\,.
\end{equation}
Solving the second block of equations in \eqref{Eq:FirstOrder} for $E_I$ yields the quadratic droop inverter voltages \eqref{Eq:InverterVoltages}. Substituting this into the first block of equations in \eqref{Eq:FirstOrder}, simplifying, and left-multiplying by $[E_L]$ yields
\begin{equation}\label{Eq:RPFETemp}
\vzeros[n] = [E_L][b_{\rm shunt}]E_L + [E_L]B_{\rm red}\left(E_L-E_L^*\right)\,,
\end{equation}
where $B_{\rm red}$ and $E_L^*$ are as in \eqref{Eq:ReducedNetworkGeneral}--\eqref{Eq:Eaverage}. The result \eqref{Eq:RPFETemp} is exactly the reduced power flow equation \eqref{Eq:ZIEquation} with $I_{\rm shunt} = \vzeros[n]$, which shows the desired correspondance. It remains only to show that this unique critical point is a minimizer. The (negative of the) Hessian matrix of $C$ is given by the matrix of coefficients in \eqref{Eq:FirstOrder}. Since the bottom-right block of this matrix is negative definite, it follows by Schur complements that the Hessian is positive definite if and only if $-(B_{\rm red} + [b_{\rm shunt}])$ is positive definite, which by assumption holds.
\end{pfof}

Note that depending on the heterogeneity of the voltage set points $E_i^*$, minimizing $\mathcal{Q}_{\rm loss} + \mathcal{Q}_{\rm load}$ may or may not conflict with minimizing $C_{\rm volt}$. The quadratic droop controller \eqref{Eq:VoltDroopQuadratic} is a primal algorithm for the optimization problem \eqref{Eq:OptCost}, and can therefore be interpreted as striking an optimal balance between maintaining a uniform voltage profile and minimizing total reactive power dissipation {\tb (cf. \cite{KT-PS-SV-MC:11,HY-DFG-SHL:12}, where similar trade-offs are studied)}. The result can also be interpreted as an application of Maxwell's minimum heat theorem to controller design \cite[Proposition 3.6]{AJvdS:10}.

\smallskip

\begin{remark}\textbf{(Generalized Objective Functions)}
To obtain the generalized feedback controller \eqref{Eq:GenQuad}, one may replace $C_{\rm volt}(E_I)$ given in \eqref{Eq:Jvolt} by
$$
C_{\rm volt}(E_I) = \sum_{i,j\in\mathcal{I}}\nolimits \frac{\alpha_{ij}}{2}E_iE_j + \beta_{ij}E_iE_j^*\,.
$$
For a general strictly convex and differentiable cost $C_{\rm volt}(E_I)$, analogous methods may be used to arrive at feedback controllers of the form $u_i = E_i \cdot \frac{\partial{C_{\rm volt}}}{\partial E_i}$. See \cite[Chapter 9]{JM-JWB-JRB:08}, \cite[Remark 4]{FD-JWSP-FB:13y}\cite{MF-RN-CC-SL:12,MF-CL-SL:13} for more on nonlinear droop curves derived from optimization methods, including dead bands and saturation. \hfill \oprocend
\end{remark}

\subsection{Power Sharing: General Case and Asymptotic Limits}
\label{Sec:PowerShare}

We now study the following question: under an incremental change in load power demands $Q_L \rightarrow Q_L + \Delta q$, how do the reactive power injections at inverters change? This is the question of \emph{power sharing}. 
For simplicity of exposition, we make two standing assumptions in this subsection. First, we consider only the case of constant power loads $Q_L(E_L) = Q_L$. Second, we consider the case of uniform inverter voltage set points, such that $E_I^* = E_{N}\cdot \vones[m]$ for some fixed nominal voltage $E_{N} > 0$. Both assumptions can be easily relaxed to yield more general (but much more cumbersome) formulae. Under these assumptions, the results of Theorem \ref{Thm:ZIP} simplify to yield the approximate load voltages
\begin{equation}\label{Eq:ConstPExpansion}
E_L \simeq E_N\left(\vones[n]-\frac{1}{E_N^2}B_{\rm red}^{-1}Q_L\right)\,,
\end{equation}
with $B_{\rm red}$ as defined in Theorem \ref{Thm:ReducedFixedPoint}. The following definition is useful for interpreting our result to follow.

\smallskip

\begin{definition}\textbf{(Differential Effective Reactance)}\label{Def:DiffReactance}
Let $E_i^{\rm oc}$ be the open-circuit voltage at load $i \in \mathcal{L}$. The \emph{differential effective reactance} $X^{\rm eff}_{ij}$ between loads $i,j \in \mathcal{L}$ is the voltage difference $E_i - E_i^{\rm oc}$ measured at load $i$ when a unit current is extracted at load $j$ with all other current injections zero.
\end{definition}

\smallskip

In other words, $X_{ij}^{\rm eff}$ is the proportionality coefficient from current injections at load $j$ to voltage deviations at load $i$. For branches $\{i,j\} \in \mathcal{E}$, we also let $X_{ij}$ denote the direct reactance between buses $i$ and $j$, and set $X_{ij} = +\infty$ if there is no branch. The main result of this section characterizes how inverter reactive power injections behave under the quadratic droop controller \eqref{Eq:VoltDroopQuadratic} as a function of the susceptance matrix $B$, load demands $Q_L$, and controller gains $K_I$.

\smallskip

\begin{theorem}\textbf{(Power Sharing)}\label{Thm:PowerShare}
Consider the closed-loop system \eqref{Eq:VoltDroopQuadraticClosed} with the approximate load voltages \eqref{Eq:ConstPExpansion}. Then the inverter reactive power injections $Q_I \in \real^{m}$ are given by 
\begin{equation}\label{Eq:SIL}
Q_I = K_I(B_{II}+K_I)^{-1}B_{IL}B_{\rm red}^{-1}Q_L\,.
\end{equation}
Moreover, the following special cases hold:
\begin{enumerate}
\item[1)] \textbf{Distance-Based Power Sharing: } In the high-gain limit where $K_i \rightarrow -\infty$ for each inverter $i \in \mathcal{I}$, it holds that
%
%
%
\begin{equation}\label{Eq:QIQLShare1}
Q_i = - \sum_{k \in \mc L} \frac{1}{X_{ik}}\sum_{j \in \mc L}X^{\rm eff}_{kj}Q_j\,,\quad i \in \mathcal{I}\,.
\end{equation}
That is, loads are preferentially supplied with power by sources which are electrically nearby;
\item[2)] \textbf{Proportional Power Sharing:} In the low-gain limit where $K_i \rightarrow 0^-$ for each $i \in \mathcal{I}$, it holds that
%
\begin{equation}\label{Eq:QIQLShare2}
Q_i = \frac{K_i}{\sum_{i \in \mathcal{I}}K_i} Q_{\rm total}\,,\quad i \in \mathcal{I}\,,
\end{equation}
where $Q_{\rm total} = \sum_{i \in \mc L}Q_i$ is the total load. That is, sources supply power based only on their relative controller gains.
\end{enumerate}
\end{theorem}

\medskip

In the high gain limit \eqref{Eq:QIQLShare1}, each inverter behaves as a stiff (constant) voltage source, and power is routed based on electrical distances. In particular, \eqref{Eq:QIQLShare1} adds up the parallel transfer paths from all loads $j$ through all intermediate loads $k$ adjacent to inverter $i$. In the low-gain limit of \eqref{Eq:QIQLShare2}, inverters react to changes in \emph{total} load by supplying reactive power in proportion to their controller gain $K_i$, which can be selected as being proportional to the generation capacity of the unit. Moreover, observe that \eqref{Eq:QIQLShare2} does not depend on the network topology. The general formula \eqref{Eq:SIL} can be interpreted as interpolating between these two regimes.


\begin{pfof}{Theorem \ref{Thm:PowerShare}}
\,From the closed-loop system \eqref{Eq:VoltDroopQuadraticClosed} we observe that in steady-state, the inverter reactive power injections are given by
\begin{equation}\label{Eq:InvInjections}
Q_I = K_I[E_I](E_I-E_N\vones[m])\,,
\end{equation}
in which the only variables are the inverter voltages $E_I$. However, Theorem \ref{Thm:ReducedFixedPoint} yielded an expression for $E_I$ in terms of $E_L$, namely \eqref{Eq:InverterVoltages}. 
Substituting the load voltages \eqref{Eq:ConstPExpansion} into \eqref{Eq:InverterVoltages} and working through some algebra, one finds that $E_I = E_N \cdot W_3\vones[m] + A_1Q_L$, where $W_3$ is row-stochastic, and hence
\begin{equation}\label{Eq:InvVoltagesConstantP}
E_I = E_N\vones[m] + A_1Q_L\,,
\end{equation}
where
\begin{align}\label{Eq:A1QShare}
A_1 &\triangleq \frac{1}{E_N}(B_{II}+K_I)^{-1}B_{IL}B_{\rm red}^{-1}\,.
\end{align}
%
Much like in the proof of Theorem \ref{Thm:ZIP}, we linearize \eqref{Eq:InvInjections} around the open-circuit operating condition $(E_I = E_N\vones[m], Q_I = \vzeros[m])$. Performing this linearization, we obtain
$$
Q_I = E_N K_I(E_I-E_N\vones[m])\,,
$$
which after substituting \eqref{Eq:InvVoltagesConstantP} and \eqref{Eq:A1QShare} yields the desired result \eqref{Eq:SIL}. To show the first statement, for $\varepsilon > 0$ define
\begin{subequations}\label{Eq:SILe}
\begin{align}
B_{\mathrm{red},\varepsilon} = B_{LL} - B_{LI}(B_{II}+\varepsilon K_I)^{-1}B_{IL}\\
S_{\varepsilon} = \varepsilon K_I(B_{II}+\varepsilon K_I)^{-1}B_{IL}B_{\mathrm{red},\varepsilon}^{-1}\,.
\end{align}
\end{subequations}
Note that for $\varepsilon = 1$, we recover the previously derived formula \eqref{Eq:SIL} as $Q_I = S_{1}Q_L$. We calculate that $\lim_{\varepsilon\rightarrow \infty} B_{\mathrm{red},\varepsilon} = B_{LL}$, and that $\lim_{\varepsilon \rightarrow \infty} \varepsilon K_I(B_{II}+\varepsilon K_I)^{-1} = I_m$. It follows that
$$
\lim_{\varepsilon\rightarrow \infty}S_{\varepsilon} = B_{IL}B_{LL}^{-1}\,.
$$
Applying Lemma \ref{Lem:Effective} then leads to the first statement. The proof of the second statement is more delicate, since for $\varepsilon = 0$, $B_{\mathrm{red},0} = B_{LL} - B_{LI}B_{II}^{-1}B_{LI}$ is singular with one-dimensional kernel spanned by $\vones[m]$ \cite[Lemma II.1]{fd-fb:11d}. Nonetheless, $B_{\mathrm{red},\varepsilon}^{-1}$ may be expanded in a Laurent series around the simple pole at $\varepsilon = 0$ \cite{PS-GWS:93}. Moreover, since $\lim_{\varepsilon \rightarrow 0} \varepsilon K_I(B_{II}+\varepsilon K_I)^{-1} = \varepsilon K_I B_{II}^{-1}$, it is clear that the only term in the Laurent series of $B_{\mathrm{red},\varepsilon}^{-1}$ giving rise to a non-zero 
contribution to $\lim_{\varepsilon\rightarrow 0}S_{\varepsilon}$ is the term proportional to $\varepsilon^{-1}$.
Some straightforward computations using \cite[Theorem 1.1]{PS-GWS:93} give that the required term in the Laurent series is $\vones[n\times n]/(\varepsilon\vones[m]^{\sf T} K_I\vones[m])$, and hence that
\begin{align*}
\lim_{\varepsilon\rightarrow 0}S_{\varepsilon} &= \lim_{\varepsilon \rightarrow 0} \varepsilon K_IB_{II}^{-1}B_{IL}\cdot \frac{1}{\varepsilon}\frac{1}{\vones[m]^{\sf T}K_I\vones[m]}\vones[n \times n]\,,\\
&= \frac{1}{\sum_{i \in \mathcal{I}}K_i}K_I B_{II}^{-1}B_{IL}\vones[m\times n]\,.
\end{align*}
Since $-B_{II}^{-1}B_{IL}$ is row-stochastic \cite[Lemma II.1]{fd-fb:11d}, $B_{II}^{-1}B_{IL}\vones[n\times n] = -\vones[m\times n]$, and the result follows by writing things out in components using Lemma \ref{Lem:Effective}. 
%
\end{pfof}

{\tb As a final technical remark regarding Theorem \ref{Thm:PowerShare}, we emphasize that the proportional power sharing result \eqref{Eq:QIQLShare2} represents a \emph{limit} under the linear approximation \eqref{Eq:ConstPExpansion}. In this low-gain limit, $B_{\rm red}^{-1}$ in \eqref{Eq:ConstPExpansion} becomes very large and for any fixed $Q_L$ the second term in the linearization \eqref{Eq:ELapprox} dominates, leading to large voltage deviations. This violates the premise under which the linearization was derived, and may lead to the loss of the closed-loop system's equilibrium point or to violation of the decoupling assumption. Thus, the extent to which the proportional power sharing result may be implemented in practice is limited by the size of the loads to be serviced, the stiffness of the grid, and the stability bottleneck associated with low load voltages. We explore this issue further via simulations in Section \ref{Sec:Sim}.
}

Nonetheless, an interesting observation can be made by comparing Theorem \ref{Thm:PowerShare} with the optimization problem \eqref{Eq:Jvolt}. Namely, large controller gains lead to accurate voltage regulation and distance-based reactive power sharing, while small controller gains minimize reactive power dissipation and yield proportional power sharing (but with risk of instability).

{\tb

\section{Simulations}
\label{Sec:Sim}

Through simulations we now demonstrate the robustness of our results to unmodeled resistive losses, coupled active/reactive power flows, and frequency dynamics. We consider the islanded microgrid of Figure \ref{Fig:ParallelandGeneral}(b) consisting of five loads and three inverters. Loads are modeled using the dynamic shunt model \eqref{Eq:DynamicShunt} for reactive power demands, along with the analogous dynamic conductance model for active power demands. The consumption set points in these load models are taken as disturbance inputs in what follows, in order to subject the system to time-varying loads. We model additional high-frequency variation on the load consumption set points as uncorrelated white noise, with a standard deviation of 15\% of the nominal set points for both active and reactive power. Branches have a mixed resistive/inductive character with non-uniform $R/X$ ratios of between 0.3 and 1.1. Inverter frequencies are controlled by the frequency droop controller
$$
\omega_i = \omega^* - m_iP_{\mathrm{e},i}\,,\qquad i \in \mathcal{I}\,,
$$
where $\omega_i$ is the inverter output frequency, $\omega^* = 2\pi\cdot 60$Hz, and $m_i > 0$ is the frequency droop coefficient.
System parameters are based on those from the experiments in \cite{jwsp-qs-fd-jmv-jmg-fb:13s}, and are omitted for brevity. The quadratic droop gains are set based on the conventional droop gains from \cite{jwsp-qs-fd-jmv-jmg-fb:13s} as $K_i = -\frac{1}{n_i E_i^*}$\@, which is a relatively stiff tuning leading to strong voltage regulation. In particular, $K_1 = K_2 = 2K_3$, meaning that inverters 1 and 2 are rated equally, with inverter 3 being rated for half as much power.

The system is under nominal loading up to $t=2$s, at which point loads $\{1,3,5\}$ are subject to sinusoidal variations with a 50\% magnitude and a 0.5s period for both active and reactive power (at constant power factor). At $t = 4$s sinusoidal variation stops and active and reactive consumption set points at load buses $\{2,4\}$ are doubled, before returning to nominal loading at $t = 6$s. Figure \ref{Fig:Sim1a} shows the resulting inverter reactive power injections.\footnote{The corresponding traces under the conventional voltage droop controller are extremely similar, and are therefore omitted for clarity.} With this tuning, the controller is able to maintain stability and track the load demands, but proportional power sharing is poor in accordance with Theorem \ref{Thm:PowerShare} (the curves for inverters 1 and 2 should overlap, but do not). Figure \ref{Fig:Sim1b} shows the resulting voltage magnitudes of both inverter and load buses. The inverter voltages stay well regulated, which prevents load buses voltages from falling even lower.

\begin{figure}[th!]
\begin{center}
\includegraphics[width=0.95\columnwidth]{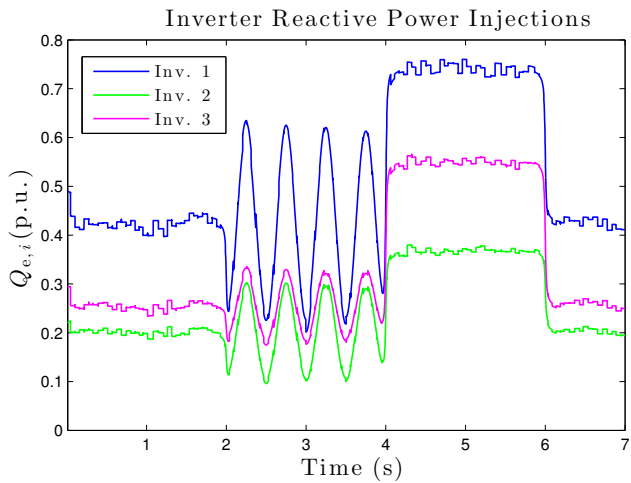}
\caption[]{Evolution of inverter reactive power injections under quadratic droop control for large feedback gains $K_i$. Vertical axis normalized by base power of 1400 VA.}
\label{Fig:Sim1a}
\end{center}
\end{figure}

\begin{figure}[th!]
\begin{center}
\includegraphics[width=0.95\columnwidth]{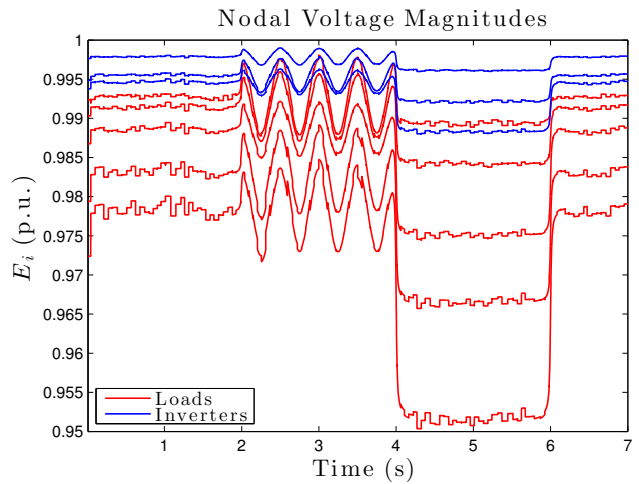}
\caption[]{Evolution of voltage magnitudes under quadratic droop control for large feedback gains $K_i$. Vertical axis normalized by base voltage of $230$V. }
\label{Fig:Sim1b}
\end{center}
\end{figure}

We now test the prediction from Theorem \ref{Thm:PowerShare} that making $K_i$ small will enforce proportional power sharing, but also make the network more prone to instability. We reduce the feedback gains $K_i$ substantially, to 5\% of their previous values, and repeat the above experiment. The resulting traces are shown in Figures \ref{Fig:Sim2a} and \ref{Fig:Sim2b}. As expected, early in the experiment reactive power sharing is enforced, with inverters 1 and 2 sharing power equally and providing twice the power of inverter 3. As seen in Figure \ref{Fig:Sim2b}, voltage magnitudes are significantly lower than they were under higher controller gains. When the load is doubled at $t = 4$s a voltage collapse occurs and the system is unable to recover. This confirms  that our theoretical results remain useful in a more general model setting.

\begin{figure}[th!]
\begin{center}
\includegraphics[width=0.95\columnwidth]{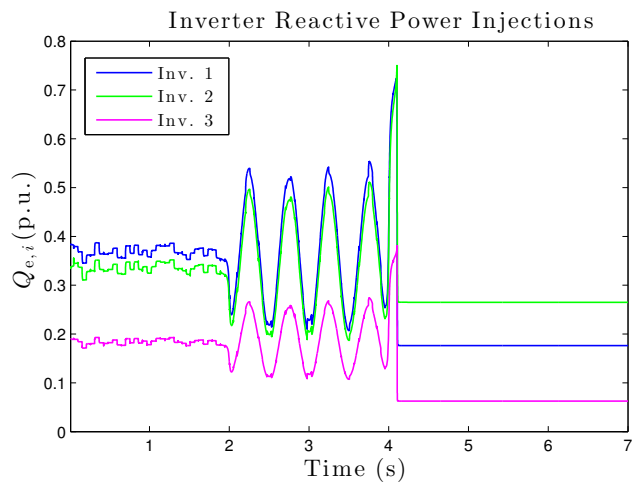}
\caption[]{Evolution of inverter reactive power injections under quadratic droop control for small feedback gains $K_i$. Vertical axis normalized by base power of 1400 VA.}
\label{Fig:Sim2a}
\end{center}
\end{figure}

\begin{figure}[th!]
\begin{center}
\includegraphics[width=0.95\columnwidth]{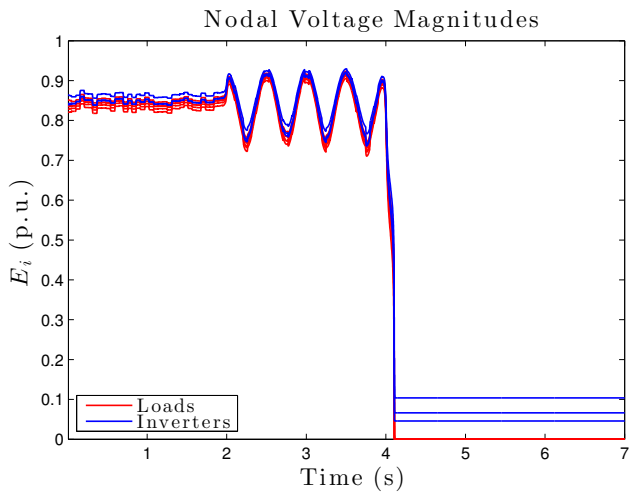}
\caption[]{Evolution of voltage magnitudes under quadratic droop control for small feedback gains $K_i$. Vertical axis normalized by base voltage of $230$V. }
\label{Fig:Sim2b}
\end{center}
\end{figure}

}

%

\smallskip

\smallskip

\section{Conclusions}
\label{Section: Conclusions}

In this work we have provided a detailed analysis of the quadratic droop controller, a tractable modification of conventional droop control for application in islanded microgrids. The special structure of quadratic droop control allows for rigorous circuit-theoretic analysis methods to be applied to the closed-loop system. A detailed analysis of system {\tb equilibria and} stability was provided for several load models, along with an optimization interpretation of the resulting decentralized controller. The power sharing characteristics of the controller were analyzed, showing that the controller interpolates between proportional power sharing (in the low-gain regime) and power sharing based on electrical distance (in the high-gain regime). Our analysis results provide easily verifiable certificates for system stability, and the proof techniques developed should prove quite useful for related microgrid {\tb analysis and} control problems.

The approach and techniques reported in this paper should serve as a guidepost for further work on the following key technical problems. While our results apply to microgrids with uniform $R/X$ ratios under a decoupling assumption, the coordinate transformation used herein to treat uniform $R/X$ ratio microgrids (see Remark \ref{Rem:Decoupling}) has the side-effect of mixing active and reactive power, making true active/reactive power sharing {\tb more subtle} to implement. Relaxing these assumptions may prove possible using recently developed analysis tools such as  {\tb coupled} 
power flow linearizations. The design of a simple, decentralized, and provably stable droop-like controller for voltage stabilization and power sharing in microgrids with non-uniform $R/X$ ratios remains an open problem. {\tb Additional important practical directions are an analysis of interactions between inner voltage/current control loops and droop control, along with a voltage stability analysis for mixed inverter/generator networks.}

While here we have addressed the problem of primary control, the authors view secondary control for islanded microgrids \cite{jwsp-qs-fd-jmv-jmg-fb:13s} as an area requiring new problem formulations and increased theoretical attention. In particular, it is not clear to what extent standard secondary control formulations such as voltage regulation and reactive power sharing (adopted from high-voltage systems and power supply backup systems, respectively) are important in sub-distribution sized microgrids. 


%
\IEEEpeerreviewmaketitle

\smallskip
\medskip
\section*{Acknowledgment}

The authors would like to thank Marco Todescato, Basilio Gentile and Sandro Zampieri for fruitful conversations.

\ifCLASSOPTIONcaptionsoff
  \newpage
\fi


\bibliographystyle{IEEEtran}
\bibliography{alias,Main,FB,New}

\appendices

\section{Supporting Lemmas and Proofs}

\begin{lemma}\textbf{(Properties of Susceptance Matrix)}\label{Lem:Properties}
The susceptance matrix $B \in \real^{(n+m)\times(n+m)}$ satisfies
\begin{enumerate}[(i)]
\item $B_{ij} = B_{ji} \geq 0$ for $i \neq j$, with strict inequality if and only if $\{i,j\} \in \mathcal{E}$\,;
\item $B$ is negative semidefinite with a simple eigenvalue at zero corresponding to the eigenvector $\vones[n+m]$\,;
\item all principal submatrices of $B$ are negative definite.
\end{enumerate}
\end{lemma}

\medskip


\begin{proposition}\textbf{(Properties of Reduced Quantities)}\label{Prop:AveragingMatrices}
The following statements hold:
\begin{enumerate}
\item[(i)] $-B_{\rm red}$ is an $M$-matrix;
\item[(ii)] $W_1$ and $W_2$ are row-stochastic;
\item[(iii)] $E_L^* > \vzeros[n]$ component-wise. 
\end{enumerate}
\end{proposition}

\begin{pfof}{Proposition \ref{Prop:AveragingMatrices}}
\,(i): The first fact follows from the closure of the set of symmetric positive definite $M$-matrices under the Schur complement, as $B_{\rm red}$ is the Schur complement of $B + \mathrm{blkdiag}(\vzeros[],K_I)$ with respect to the $B_{II} + K_I$ block \cite[Lemma II.1]{fd-fb:11d}. 
Similar to \cite[Lemma II.1]{fd-fb:11d}, it can be verified that $W_{1}\vones[m] = \vones[n]$ and $W_{2}\vones[n+m] = \vones[m]$. The proof of (iii) then follows from (ii).
\end{pfof}


\begin{pfof}{Theorem \ref{Thm:ZIP}}
\,With the ZIP load model \eqref{Eq:ZIP}, the reduced power flow equation \eqref{Eq:ZIEquation} is modified to
\begin{equation}\label{Eq:RPFEZIP}
Q_L = {-[E_L](B_{\rm red}+[b_{\rm shunt}])(E_L-E_L^{\rm ZI})}\,.
\end{equation}
Consider now the invertible change of coordinates defined by
\begin{equation}\label{Eq:CoordinateChangeEL}
E_L = [E_L^{\rm ZI}](\vones[n]-Q_{\rm sc}^{-1}Q_L+\epsilon)\,,
\end{equation}
where $\epsilon \in \real^n$ is our new variable. Substituting this into the power flow \eqref{Eq:RPFEZIP}, some basic algebra leads to
\begin{equation}\label{Eq:hexpression}
h(\epsilon,q_L) \triangleq Q_{\rm sc}\epsilon + [q_L-\epsilon]Q_{\rm sc}(q_L-\epsilon) = \vzeros[n]\,,
\end{equation}
where we have abbreviated $q_L = Q_{\rm sc}^{-1}Q_L$. First, note that $h(\vzeros[n],\vzeros[n]) = \vzeros[n]$. Next, an easy calculation gives that $\frac{\partial h}{\partial \epsilon}(\vzeros[n],\vzeros[n]) = Q_{\rm sc}$, which is full rank. It follows from the Implicit Function Theorem that there exist open sets $U_0, V_0 \subset \real^n$ (each containing the origin) and a continuously differentiable function $\map{H}{U_0}{V_0}$ such that $h(H(q_L),q_L) = \vzeros[n]$ for all $q_L \in U_0$. This in turn shows that for $\|q_L\| = \|Q_{\rm sc}^{-1}Q_L\|$ sufficiently small, the expression \eqref{Eq:CoordinateChangeEL} solves the reduced power flow equation \eqref{Eq:RPFEZIP}.
To obtain the quadratic bound on $\|\epsilon\|$, rearrange \eqref{Eq:hexpression} and take the norm of both sides to obtain
\begin{equation}\label{Eq:Qnorm}
\|Q_{\rm sc}\epsilon\| = \|[q_L-\epsilon]Q_{\rm sc}(q_L-\epsilon)\|\,.
\end{equation}%
Since $Q_{\rm sc}$ is invertible, the left-hand side of \eqref{Eq:Qnorm} is lower-bounded by $\alpha \|\epsilon\|$, where $\alpha > 0$. Moreover, a series of elementary inequalities shows that the right-hand side of \eqref{Eq:Qnorm} is upper bounded by $\beta \|q_L - \epsilon\|^2 \leq \beta (\|q_L\| + \|\epsilon\|)^2$, where $\beta \geq \alpha$ depends only on $Q_{\rm sc}$. Combining these bounds, we find that $\epsilon$ satisfies 
\begin{equation}\label{Eq:EpsIneq}
\|\epsilon\| \leq \frac{\beta}{\alpha}(\|q_L\| + \|\epsilon\|)^2\,.
\end{equation}
Since $\|q_L\|$ is sufficiently small, in particular it holds that $4\alpha \|q_L\|/\beta < 1$. In this case, some simple algebra shows that inequality \eqref{Eq:EpsIneq} further implies that
\begin{align*}
\|\epsilon\| &\leq \frac{\alpha}{2\beta} - \|q_L\| - \frac{\alpha}{2\beta}\sqrt{1-4\alpha \|q_L\|/\beta}\\
&\leq \frac{\alpha}{2\beta} - \|q_L\|- \frac{\alpha}{2\beta}\left(1-\frac{\beta}{\alpha}\|q_L\|-8\|q_L\|^2\frac{\beta^2}{\alpha^2}\right)\\
&= 4\frac{\beta}{\alpha}\|q_L\|^2 = \mathcal{O}(\|Q_{\rm sc}^{-1}Q_L\|^2)\,,
\end{align*}
where we have used the fact that for $a \in [0,1]$, $\sqrt{1-a} \geq 1-a/2 - a^2/2$. Exponential stability follows immediately by continuity from the case of ZI loads. 
\end{pfof}


\begin{pfof}{Theorem \ref{Thm:StabDS}} We provide only a sketch of the proof. In steady state, the DS model \eqref{Eq:DynamicShunt} is equivalent to a constant-power model, and hence the approximate steady-state load voltages $E_L^{\rm DS}$ are given as in Theorem \ref{Thm:ZIP} with ZI components $b_{\rm shunt} = \vzeros[n]$ and $I_{\rm shunt} = \vzeros[n]$, yielding the stated equilibrium.
Stability can no longer be shown by applying Theorem \ref{Thm:Stab}, since Theorem \ref{Thm:Stab} was proven in the absence of load dynamics. A tedious but straightforward calculation shows that applying the proof methodologies of Theorems \ref{Thm:Stab} and \ref{Thm:ZIP} to the extended dynamics \eqref{Eq:VoltDroopQuadraticClosed},\eqref{Eq:DynamicShunt} yields local exponential stability.
\end{pfof}


\begin{lemma}\textbf{(Electrical Distances)}\label{Lem:Effective} Let $X^{\rm eff}_{kj}$ be the differential effective reactance between loads $k,j \in \mathcal{L}$, and for $\{i,k\} \in \mathcal{E}$ let $X_{ik} = 1/B_{ik}$ be the direct reactance, with $X_{ik} = +\infty$ if $\{i,k\} \notin \mathcal{E}$. Then
$$
-(B_{IL}B_{LL}^{-1})_{ij} = \sum_{k \in \mathcal{L}}\nolimits X^{\rm eff}_{kj}/X_{ik}\,.
$$
\end{lemma}

\begin{pfof}{Lemma \ref{Lem:Effective}}
From the current balance equations $-\boldsymbol{\mathrm{j}}I_L = B_{LL}E_L + B_{LI}E_I$, observe by setting $I_L = \vzeros[n]$ that the open-circuit load voltages are given by $E_L^{\rm oc} = -B_{LL}^{-1}B_{LI}E_I$. Then the current balance can be rewritten as $E_L - E_L^{\rm oc} = -\boldsymbol{\mathrm{j}}B_{LL}^{-1} I_L$, and it follows that $X^{\rm eff} = -B_{LL}^{-1}$. The desired formula is then immediate by noting that for $i \in \mathcal{L}$, $j \in \mathcal{I}$, and $\{i,j\} \in \mathcal{E}$, $1/(B_{LI})_{ij} = X_{ij}$ by definition.
\end{pfof}

%

\vspace{-2em}

\begin{IEEEbiography}[{\includegraphics[width=1in,height=1.25in,clip,keepaspectratio]{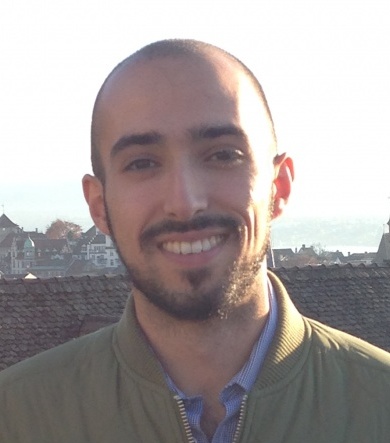}}]{John W. Simpson-Porco} (S'11--M'16) received the B.Sc. degree in engineering physics from QueenÕs University, Kingston, ON, Canada in 2010, and the Ph.D. degree in mechanical engineering from the University of California at Santa Barbara, Santa Barbara, CA, USA in 2015.

He is currently an Assistant Professor of Electrical and Computer Engineering at the University of Waterloo, Waterloo, ON, Canada. He was previously a visiting scientist with the Automatic Control Laboratory at ETH Z\"{u}rich, Z\"{u}rich, Switzerland. His research focuses on the control and optimization of multi-agent systems and networks, with applications in modernized power grids.

Prof. Simpson-Porco is a recipient of the 2012-2014 Automatica Prize, the 2014 Peter J. Frenkel Foundation Fellowship, and the Center for Control, Dynamical Systems and Computation Outstanding Scholar Fellowship.
\end{IEEEbiography}

\vspace{-2em}

\begin{IEEEbiography}[{\includegraphics[width=1in,height=1.25in,clip,keepaspectratio]{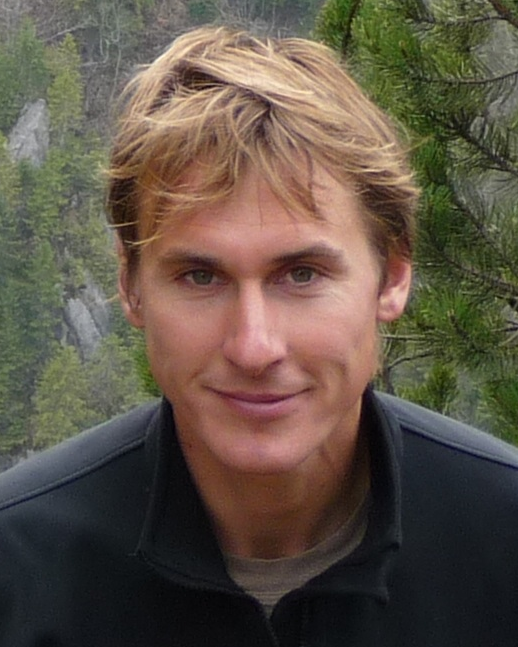}}]
{Florian D\"{o}rfler} (S'09--M'13) is an Assistant Professor at the Automatic Control Laboratory at ETH Z\"{u}rich. He received his Ph.D. degree in Mechanical Engineering from the University of California at Santa Barbara in 2013, and a Diplom degree in Engineering Cybernetics from the University of Stuttgart in 2008. From 2013 to 2014 he was an Assistant Professor at the University of California Los Angeles. His primary research interests are centered around distributed control, complex networks, and cyberÐphysical systems currently with applications in energy systems and smart grids. He is a recipient of the 2009 Regents Special International Fellowship, the 2011 Peter J. Frenkel Foundation Fellowship, the 2010 ACC Student Best Paper Award, the 2011 O. Hugo Schuck Best Paper Award, and the 2012-2014 Automatica Best Paper Award. As a co-advisor and a co-author, he has been a finalist for the ECC 2013 Best Student Paper Award.
\end{IEEEbiography}

\vspace{-3em}


\begin{IEEEbiography}[{\includegraphics[width=1in,height=1.25in,clip,keepaspectratio]{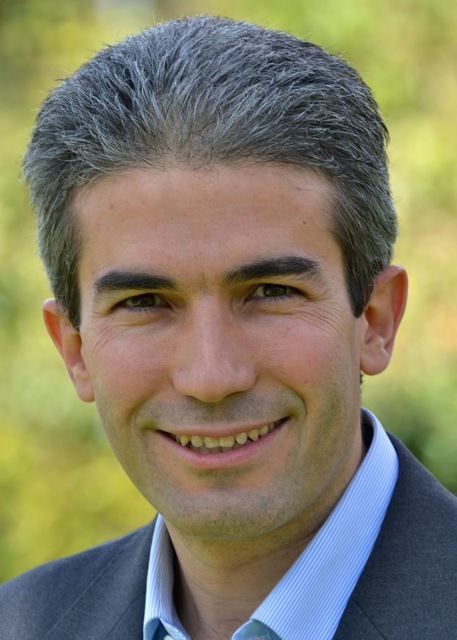}}]{Francesco Bullo} (S'95--M'99--SM'03--F'10) received the Laurea degree "summa cum laude" in Electrical Engineering from the University of Padova, Italy, in 1994, and the Ph.D. degree in Control and Dynamical Systems from the California Institute of Technology in 1999. He is a Professor with the Mechanical Engineering Department and the Center for Control, Dynamical Systems and Computation at the University of California, Santa Barbara. He
was previously associated with the University of Illinois at Urbana-Champaign.
His main research interests are network systems and distributed control
with application to robotic coordination, power grids and social networks.
He is the coauthor of "Geometric Control of Mechanical Systems" (Springer,
2004, 0-387-22195-6) and "Distributed Control of Robotic Networks"
(Princeton, 2009, 978-0-691-14195-4).  He received the 2008 IEEE CSM
Outstanding Paper Award, the 2010 Hugo Schuck Best Paper Award, the 2013
SIAG/CST Best Paper Prize, and the 2014 IFAC Automatica Best Paper
Award. He has served on the editorial boards of IEEE Transactions on
Automatic Control, ESAIM: Control, Optimization, and the Calculus of
Variations, SIAM Journal of Control and Optimization, and Mathematics of
Control, Signals, and Systems.
\end{IEEEbiography}

\end{document}